\documentclass[b4paper,12pt]{amsart}

\usepackage{mystyle_english_article_10}
\usepackage{anysize}

\renewcommand{\structure}{}
\renewcommand{\corange}{}
\renewcommand{\clrgreen}{}

\newenvironment{nouppercase}{%
  \renewcommand{\uppercasenonmath}[1]{}}{}

\newcommand{\secref}[1]{Section~\ref{#1}}
\newcommand*{\affaddr}[1]{#1} % No op here. Customize it for different styles.
\newcommand*{\affmark}[1][*]{\textsuperscript{#1}}
\renewcommand{\cred}{}

\begin{document}

\title[Multifractal Analysis Modeling TCP CUBIC]
{Large Deviation Multifractal Analysis of a Process Modeling TCP CUBIC}

\author{K\'aroly~Simon\affmark[1], S\'andor~Moln\'ar\affmark[2], J\'ulia Komj\'athy\affmark[3], and P\'eter~M\'ora\affmark[1]\\
\affaddr{\affmark[1]Department of Stochastics, Institute of Mathematics\\ Budapest University of Technology and Economics, Hungary}\\
\affaddr{\affmark[2]Department of Telecommunications and Media Informatics\\ Budapest University of Technology and Economics, Hungary}\\
\affaddr{\affmark[3]Stochastic Section, Eindhoven University of Technology, The Netherlands}\\
%\email{\{A,B,C,D,E\}@university.edu}\\
%\affaddr{\LaTeX\ University}
}

\begin{abstract}
Multifractal characteristics of the Internet traffic have been discovered and discussed in several research papers so far. However, the origin of this phenomenon is still not fully understood. It has been proven that the congestion control mechanism of the Internet transport protocol, i.e., the mechanism of TCP Reno can generate multifractal traffic properties. Nonetheless, TCP Reno does not exist in today's network any longer, surprisingly traffic multifractality has still been observed. In this paper we give the theoretical proof that TCP CUBIC, which is the default TCP version in the Linux world, can generate multifractal traffic.  We give the multifractal spectrum of TCP CUBIC traffic and compare it with the multifractal spectrum of TCP Reno traffic. Moreover, we present the multifractal spectrum for a more general model, where TCP CUBIC and TCP Reno are special cases. Our results also show that TCP CUBIC produces less bursty traffic than TCP Reno.
\end{abstract}

\begin{nouppercase}
\maketitle
\end{nouppercase}

%%%%%%%%%%%%%%%%%%%%%%%
%%%%%%%%%%%%%%%%%%%%%%%
%%%%%%%%%%%%%%%%%%%

% a vége a

\section{Introduction}
\label{sec:intro}

The extensive measurements and analyses of network traffic in the previous three decades have revealed rich and complex traffic properties highlighting \textit{scale invariant features} and \textit{fractal characteristics}. These properties have helped to understand the most striking feature of network traffic: its \textit{burstiness}. Burstiness refers to the inherent nature of network traffic meaning that packets are transmitted in short uneven spurts. A kind of burstiness manifests itself over long periods identified as \textit{self-similarity} and \textit{long-range dependence (LRD)}, which have been studied intensively since their first discovery in the research of Leland et al.~\cite{leland1993self}. LRD can be captured well by \textit{monofractal} models like Fractional Brownian Motion (FBM)~\cite{riedi2000toward}. \textit{Temporal burstiness}, which is the variation of traffic intensity on small time-scales, has also been explored. However, comprehensive research has shown that simple monofractal scaling cannot describe traffic burstiness at this scale and more sophisticated \textit{multifractal} models are needed~\cite{feldmann1998data,riedi2000toward}.

\textit{Multifractal analysis} has been found to be a useful tool to explore temporal burstiness and multifractal characteristics of network traffic at small time scales~\cite{riedi2000toward}. Several research studies have revealed that the multifractality nature of traffic is mainly due to the TCP (Transmission Control Protocol)~\cite{Reidi1997, riedi2000toward}, which carries more than 80\% of network traffic~\cite{TCP_Volume}. Moreover, it has also been shown that there are significant performance implications of the multifractality of network traffic regarding queueing performance~\cite{Dand2003,erramilli2000performance}. As a consequence, it is vital to explore the behavior of TCP traffic in order to characterize its multifractal features and this is exactly the motivation for our work.

Our methodology in this paper has been inspired by the work of L\'{e}vy-V\'{e}hel and Rams~\cite{Rams2012}, where a Large Deviation Multifractal Analysis has been carried out. Their work focuses on the analysis of a simplified TCP model of TCP Reno and presents its multifractal spectra. This important result has brought forth theoretical proof that TCP Reno dynamics itself (i.e., the additive increase multiplicative decrease (AIMD) mechanism) can lead to multifractal behavior.

In this paper we make a step further in discovering the \textit{multifractal nature of network traffic}. We have been motivated by the suprising fact that TCP Reno does not exist in today's network any longer traffic multifractality  has still been observed. Compared to the TCP Reno model in~\cite{Rams2012} we choose a realistic model for today's network, i.e., the TCP CUBIC, which is the default TCP version in the Linux world and analyze its\textit{ multifractal spectrum}. Moreover, we also present the multifractal spectrum for a more \textit{general case} where the functions used in TCP CUBIC and TCP Reno are special cases. We also compare our results to the results in~\cite{Rams2012} and provide the \textit{theoretical proof of showing that TCP CUBIC generates less bursty traffic than TCP Reno}.

\cred{Our proves follow the line   of the proofs of L\'{e}vy-V\'{e}hel  and Rams paper \cite{Rams2012}. However, in this much more general setup, the details became much more cumbersome and we needed to face with technical difficulties that did not appear in the case of TCP Reno.}

The rest of the paper is organized as follows. \secref{sec:relwork} discusses the related work in the field of TCP multifractality. In~\secref{sec:cubicmodel}, we present the TCP CUBIC model with its multifractal spectrum. The comparison of the multifractal spectrum of TCP CUBIC and TCP Reno is discussed in ~\secref{sec:comparison}. In~\secref{sec:generalization} we present the more general model and its multifractal spectrum, with TCP CUBIC being a special case. The detailed proof of our results is given in ~\secref{sec:stationarity}-\secref{z104}. Finally, \secref{sec:implications} concludes the paper.

\section{Related Work}
\label{sec:relwork}

Traffic burstiness has been investigated for a long time in the teletraffic research~\cite{frost1994traffic,jagerman1997stochastic,abry2002multiscale} and found to be one of the key characteristics of network traffic from network design, dimensioning and performance evaluation point of view. Several advanced burstiness measures have been proposed because simple measures like Peak-To-Mean ratio (PMR) and Coefficient of Variation were found to be inadequate~\cite{frost1994traffic,jagerman1997stochastic,abry2002multiscale}. These measures includes, for example, Hurst parameter, Index of dispersion for intervals (IDI), Index of dispersion for counts (IDC), the peakedness functional, etc. see~\cite{frost1994traffic,jagerman1997stochastic,abry2002multiscale}. In addition, recently it has been found that burstiness is even more complex and requires a full high-order and correlation characterization and a more appropriate burstiness characterization was proposed via the multifractal analysis of the traffic~\cite{abry2002multiscale,riedi2000toward}. In our paper we chose the multifractal analysis of TCP CUBIC traffic and the first time gave a mathematical proof based on the multifractal spectrum why TCP CUBIC burstiness is smaller than TCP Reno. The only related paper in this regard is~\cite{cai2007stochastic} but in this paper both the goal and the methodology are different, i.e., focusing on the rate variation metric by convex ordering and measured by the Coefficient of Variation capturing only the second order properties of traffic variability. In contrast, we focus on a much richer characterization of burstiness by the Large Deviation Multifractal Analysis and measured by the full multifractal spectrum.

Multifractal characteristics of network traffic was first published by Riedi and L\'{e}vy-V\'{e}hel~\cite{Reidi1997,Reidi1997b}. From this discovery several research studies have been carried out to understand the multiscale nature of network traffic, see~\cite{abry2002multiscale} for an excellent overview. From traffic modeling purpose different model classes have been developed, e.g., \textit{multiplicative cascades}~\cite{feldmann1998data,riedi2002}, \textit{Fractional Brownian Motion in multifractal time}~\cite{Reidi1997b}, \textit{$\alpha$-stable processes}~\cite{samoradnitsky1994stable} and other general multifractal models~\cite{molnar2001general}.

%A promising modeling approach to use the most well-known member of the class of mul-
%tifractal processes, i.e., the \textit{multiplicative cascades}~\cite{feldmann1998data,riedi2002}. A simplified version of this model is based on the binary tree structure yields to the \textit{binomial cascade}~\cite{riedi2002}. An interesting approach to combine this process with the Fractional Brownian Motion~\cite{riedi2000toward} resulted in a new model class called the \textit{Fractional Brownian Motion in Multifractal Time}~\cite{Reidi1997b} having the appealing characteristics with capturing both the Long-Range Dependent and multifractal scaling independently. Another model family is the \textit{$\alpha$-stable process}~\cite{samoradnitsky1994stable}, where the high-order statistics are not finite and it creates a process with irregular multifractal structure. A simple case of this process is the \textit{Linear Fractional Stable Motion}~\cite{samoradnitsky1994stable}. Another general multifractal model is based on the pairwise product of a multiplicative cascade and an independent lognormal process\cite{molnar2001general}.

Regarding the reason \textit{why} multifractality is present and observed in the Internet we still have a lack of clear understanding. In order to find explanations for the traffic multifractality Feldmann et al.~\cite{feldmann1998data} presented a cascade framework that allowed for a plausible physical explanation of the observed multifractal scaling
properties of network traffic. They applied wavelet-based analysis and obtained a detailed description of multifractality. Their main findings is that the cascade paradigm over small time scales appears to be a traffic invariant for wide area network traffic.

However, there is no physical evidence that TCP traffic actually behaves
as a cascading or multiplicative process. L\'{e}vy-V\'{e}hel and Rams~\cite{Rams2012} showed that adding sources managed by TCP can lead to multifractal behavior. This result demonstrates that there is no need for assuming any multiplicative structure but multifractality is simply due to the interactions of additive increase multiplicative decrease (AIMD) mechanism of TCP and to the random non-synchronous transmission of the sources.
The result was proved for a simplified TCP model capturing the main features of TCP Reno. 

Nevertheless, TCP Reno is not used in recent networks and due to the inefficient performance of old TCPs (e.g., TCP Reno or TCP Vegas) a fast development of several new TCP versions have been triggered (e.g., FAST~\cite{FAST-2}, HSTCP~\cite{RFC3649}, STCP~\cite{Kelly03}, BIC~\cite{BIC}, etc.). Among these versions the TCP CUBIC~\cite{CUBIC} is widespread since it is implemented and used by default in Linux kernels 2.6.19 and above. This was the main motivation for us focusing on TCP CUBIC in this paper.

%In this paper we consider TCP CUBIC as the TCP version we want to model since this is the version which practically governs the Internet. We use the method by L\'{e}vy-V\'{e}hel and Rams~\cite{Rams2012} and extend it to model this relevant TCP version in order to capture its multifractal features by multifractal analysis. \cite{Crovella1996self}

%In this paper we consider TCP CUBIC as the TCP version we want to model since this is the version which practically governs the Internet. We use the method by L\'{e}vy-V\'{e}hel and Rams~\cite{Rams2012} and extend it to model this relevant TCP version in order to capture its multifractal features by multifractal analysis. \cite{Crovella1996self}

\section{Multifractal Analysis of TCP CUBIC}%\label{t198}
\label{sec:cubicmodel}
In this Section we introduce the model of TCP CUBIC with some basic definitions and properties used in the analysis. Furthermore, we present our results for the multifractal spectrum of TCP CUBIC.

\subsection{The Model of TCP CUBIC}\label{z14}
TCP CUBIC is a successful transport protocol in the evolution of TCP versions where the congestion control method is optimized to high bandwidth networks~\cite{CUBIC}. TCP CUBIC is similar to the standard TCP Reno algorithm regarding the additive increase and multiplicative decrease behavior but there are also major differences. For instance, TCP CUBIC increases it's sending rate according to a \textit{cubic function}~\cite{CUBIC} instead of linear increase, that was implemented in TCP Reno. In the following model we have captured the main characteristics of TCP CUBIC.

The aggregated TCP CUBIC traffic is modelled by the infinite sum of independent random functions
\begin{equation}\label{t199}
  Z(t):=\sum\limits_{j=1}^{\infty }Z_j(t),
\end{equation}
where $Z_j(t)$ are piecewise deterministic functions on random time intervals representing the TCP CUBIC traffic from source $j$. The main idea is the following: events of losses of packages occur at a sequence of random points in time, that we denote by a random sequence $T_1^j, T_2^j, \dots$ for each $j\in \mathbb N$.
 Then, $Z_j(t)$ is deterministic and monotone increasing on each random time interval $\left[T_{k-1}^j,T_{k}^{j}\right)$ for $k=1,2, \dots $. This corresponds to the fact that the TCP protocol increases its sending rate when no loss occurs.
 First we define the random time of losses $\left(T_{k}^{j}\right)_{k=1,2, \dots}$,  then we describe the deterministic rule which gives $Z_j(t)$ on $\left[T_{k-1}^j,T_{k}^{j}\right)$.

Each $Z_j(t)$ has an intensity parameter $\lambda_j$. \cred{Without loss of generality we may} assume that
$
  1  \leq  \lambda_1<\lambda_2<\cdots,
$
 and we also assume that the sequence  $\left(\lambda_j\right)_{j=1}^{\infty }$ is regular (see Definition \ref{t220}).
%Each $Z_j(t)$ has an intensity parameter $\lambda_j$. We assume that $\la_1\ge 1$ and the sequence $\{\la_j\}_{j\in \N}$ is increasing, and is regular, as defined in Definition~\ref{t220} below.
Moreover, we require that
   \begin{equation}\label{o59}
\sum\limits_{j=1}^{\infty }\left(\frac{\log\lambda_j}{\lambda _{j}}\right)^{3}<\infty .
\end{equation}
Then the sequence of losses,  $\left(T_{k}^{j}\right)_{k=0}^{\infty } $ are independent $\mathrm{Poisson}(\lambda_j)$ processes for different $j$'s. Hence,
\begin{equation}\label{t197}
  T_{k}^{(j)}:=\sum\limits_{j=0}^{k}\tau_{k}^{(j)},
\end{equation}
 where the inter-event times
$\left(\tau_{k}^{(j)}\right)_{k=1}^{\infty }
$
  are $\mathrm{Exp(\lambda_j)}$ random variables
such that $\left(\tau_{k}^{(j)}\right)_{j,k}$ are independent.

We will define the functions $Z_j(t)$ in a right-continuous way. First we define
 $Z_j(0)$ in an arbitrary way such that $\sum\limits_{j=1}^{\infty }Z_j(0)<\infty$. Then we assume that $Z_j(t)$ is already defined for all $t<T_{k-1}^{(j)}$.
 So, in particular $Z_j(T^{(j)-}_{k-1})$, the left hand side limit of $Z_j$ at
$T^{(j)}_{k-1}$ has already been defined.  (See Figure \ref{z2}.) Then
 we define $Z_j(t)$ for a $t\in\left[T_{k-1}^{(j)},T_{k}^{(j)}\right)$
by
\begin{equation}\label{o58}
  \large{Z_j(t):=g_{Z_j(T^{(j)-}_{k-1})}\left(t-T^{(j)}_{k-1},\right)},
\end{equation}
where
\begin{equation}\label{o84}
g_w(t):=C\left(t-\sqrt[3]{\frac{wb}{C}}\right)^3+w,
\end{equation}
 where we set $b=0.7$ and $C=0.4$ in the TCP CUBIC model since this is the setting in most of the Linux kernels, see Figure~\ref{t195}.
\begin{figure}[H]
  \centering
 \includegraphics[width=6cm]{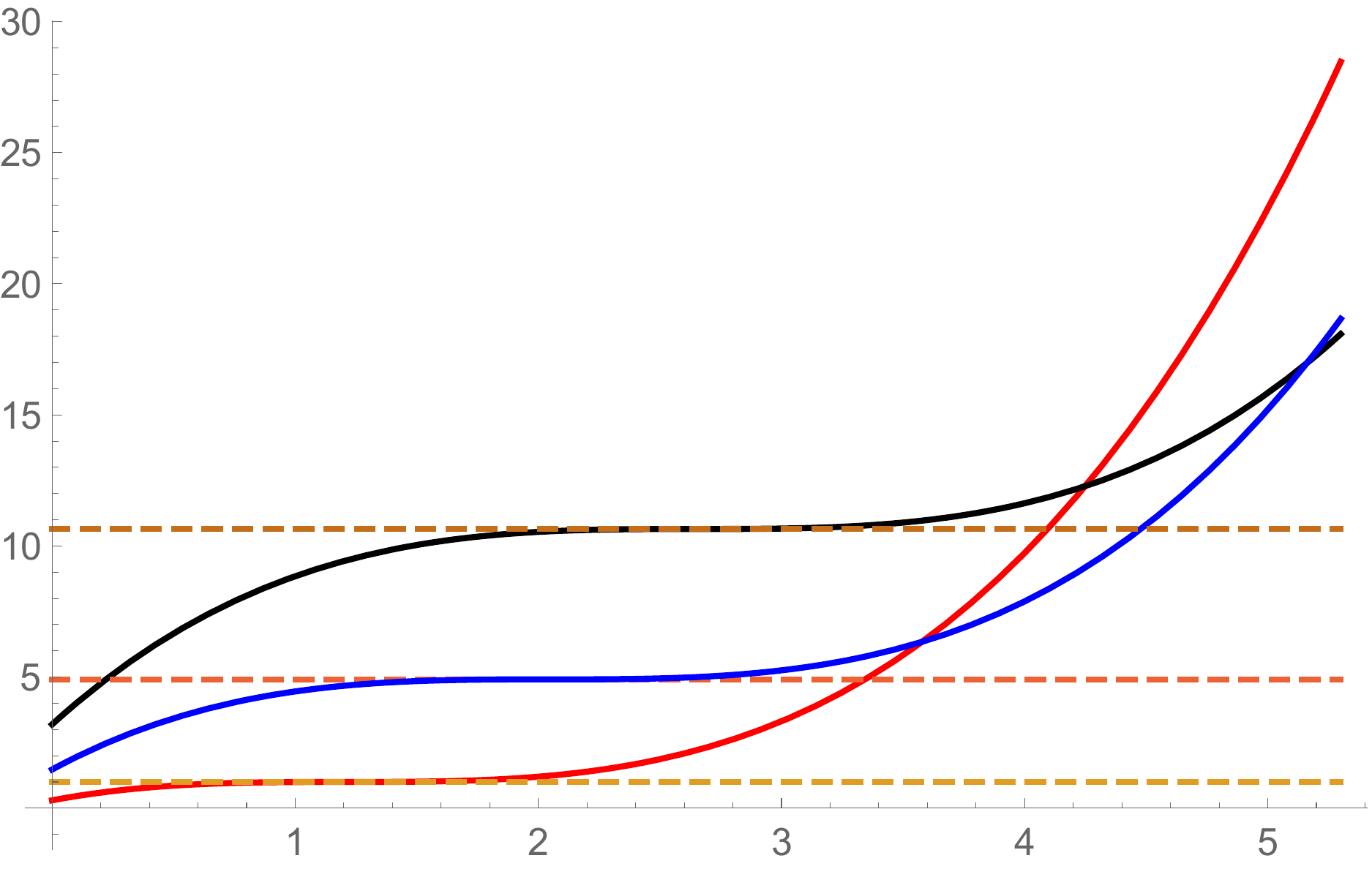}
\\
  \caption{We compare $g_1(t)$, $g_{1.7}(t)$ and $g_{2.3}(t)$ in the TCP CUBIC model.}\label{t195}
\end{figure}
One can see that the family $\left\{g_w(\cdot)\right\}_{w>0}$ of functions have self-affine property:
\begin{equation}\label{o85}
 \large{g_w(t)=\lambda ^3g_{\frac{w}{\lambda ^{3}}}\left(\frac{t}{\lambda }\right)}.
\end{equation}
Set $\lambda_0:=1$ and we define the reference process $Z_0$ by \eqref{o58}, see Figure~\ref{z2}.
\begin{figure}[H]
  \centering
  \includegraphics[width=7cm]{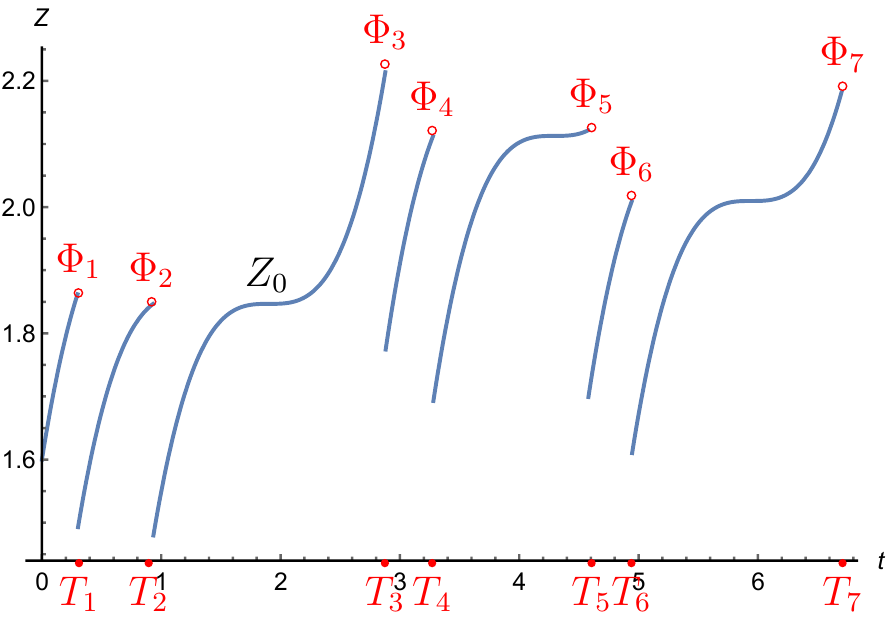}\\
  \caption{A realization of  $Z_0$. The value of
  $Z_0(t)$ immediately before the $k$-th loss is $\Phi_k:=Z_0(T^{(0)-}_k)$.
  }\label{z2}
\end{figure}
We define $\Phi_k:=Z_0(T^{(0)-}_k)$, the value of $Z_0(t)$ right before the $k$th loss happens. It is an elementary calculation to see that the size of the $k$th loss is then $b \Phi_k$, that is, $Z_0(T^{(0)}_k)=(1-b)\Phi_k$. Due to the fact that the function $g_w$ is nonlinear, $Z_0(t)$ is not a Markov process. However, $\{\Phi_k\}$ is a Markov chain.

\subsection{Large Deviation Multifractal Spectrum $f_g(\alpha)$}\label{z28}
We define the Large Deviation Multifractal Spectrum, the main object of our analysis. Recall the function $Z(t)$ from \eqref{t199}.
\clrgreen{\begin{definition}[Large deviation multifractal sepectrum]\label{t205}
First in  we define the Large deviation multifractal spectrum for increments then  for the osscilations.
We write $I_{\ell }^{k}:=\left[\frac{k}{2^\ell} ,\frac{k+1}{2^\ell }\right]$ for the $k$-th level-$\ell $ binary intervals, $k=0, \dots  ,2^{\ell}-1$.
\begin{enumerate}
  \item  The $k$th level-$\ell$ increment of the random function
$Z$ is
\begin{equation}\label{z3}
  \Delta_{\ell }^{k}Z:=Z\left(\frac{k+1}{2^\ell }\right)-
  Z\left(\frac{k}{2^\ell }\right)
\end{equation}
We define $\Delta_{\ell }^{k}Z_j, j\ge 1$ in an analogous way.
We denote the number of level-$\ell $ binary intervals on which
 the increment of $Z$ is approximately $\left(1/2^{\ell }\right)^\alpha$ by
 $N_{\ell }^{\varepsilon }(\alpha)$:
 \begin{equation}\label{o121}
N_{\ell }^{\varepsilon }(\alpha):=
\# J_{\alpha,\ell,\varepsilon },
\end{equation}
where
\begin{equation}\label{z27}
 J_{\alpha,\ell,\varepsilon }:=
\left\{
0 \leq k<2^\ell :
|\Delta_{\ell }^{k }Z|\in\left(
\frac{1}{2^{\ell (\alpha+\varepsilon)}}
,
\frac{1}{2^{\ell (\alpha-\varepsilon)}}
\right)
\right\}
\end{equation}
In general $N_{\ell }^{\varepsilon }(\alpha)\ll 2^{\ell }$.
The large deviation multifractal spectrum $f_g(\alpha)$ is the smallest exponent for which for  (in a very vague sense )
$$
N_{\ell }^{\varepsilon }(\alpha) \precsim \left(2^\ell \right)^{f_g(\alpha)}.
$$
More precisely,
\begin{equation}\label{t212}
f_g(\alpha ):=\lim\limits_{\varepsilon \to 0}
\limsup\limits_{\ell \to \infty }
\frac{\log N_{\ell }^{\varepsilon }(\alpha)}{\log 2^\ell }.
\end{equation}
  \item The large deviation multifractal spectrum for osscilation $f_g^{\mathrm O}(\alpha)$, is determined by the \emph{oscillations of} $Z$.
That is, in Definition \ref{t205}, we change $\Delta^k_\ell Z$ in \eqref{z3} to
\begin{equation*}\label{osc-def}
\mathrm O^k_\ell Z:= \sup_{x \in I^k_\ell} Z(x) - \inf_{x \in I^k_\ell} Z(x),
\end{equation*}
and then define $J_{\alpha, \ell, \varepsilon}^{\mathrm{O}}$ and consecutively $N_\ell^{\varepsilon, \mathrm{O}}(\alpha)$ similarly as in \eqref{z27} and \eqref{o121} by replacing $\Delta^k_\ell Z$ by $\mathrm{O}_\ell^k Z$.
Finally, we obtain $f_g^{\mathrm{O}}(\alpha)$ as in \eqref{t212} but with $N_\ell^{\varepsilon, \mathrm{O}}(\alpha)$ instead of $N_\ell^{\varepsilon}(\alpha)$.
\end{enumerate}
\end{definition}}

\begin{figure}
  \centering
  \includegraphics[height=5cm]{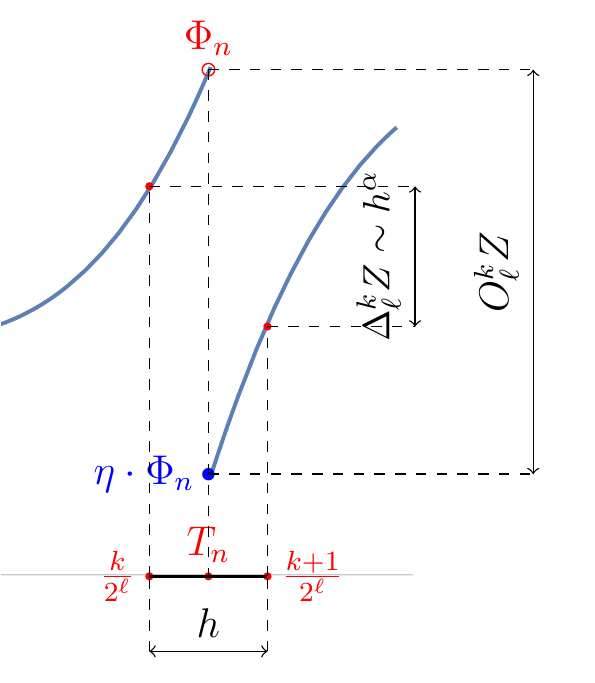}
  \caption{The increment and the oscillation, where $0<\eta<1$ is a constant  (defined in the supplement).}\label{z102}
\end{figure}

\subsection{Blumenthal-Getoor Index and the Regularity of $\left(\lambda_j\right)_{j=1}^{\infty }$}\label{z102}
The notation here are used later in the more general setting. Let $\theta \geq 1$ be an arbitrary real number.  (In the case of TCP CUBIC $\theta=3$.)
Blumenthal-Getoor index of an infinite sequence $(\lambda_j)_{j\ge1}$ and exponent $\theta$ is defined as
\begin{equation}\label{z4}
  \beta_\theta :=1+\inf
\left\{
\gamma \geq 0:
\sum\limits_{j}\frac{1}{\lambda _{j}^{\theta \cdot \gamma }}
\right\}
<\infty .
\end{equation}
Clearly $\beta_\theta$ depends on $\left(\lambda_j\right)_{j\ge 1}$ and $\theta$.
It follows from \eqref{o59} that
\begin{equation*}%\label{t201}
  \beta_\theta\in\left(1,2\right]
\end{equation*}
To obtain an equivalent definition first we fix a natural number $L>2$. The powers of $L$ are denoted sometimes by $L_k:=L^k$. We write

\begin{equation}\label{def:mk-nk}
M_k:=\left\{\#\left\{i:\lambda^\theta _i<L^k\right\}\right\} \mbox{ and }
N_k:=\#\left\{i:\lambda _{i}^{\theta}\in (L^{k-1},L^k]\right\}.
\end{equation}
 It is easy to check that
\begin{equation}\label{o80}
  \beta_\theta=1+\limsup\limits_{n\to\infty }
\frac{\log N_k}{k\log L} .
\end{equation}
 By the definitions \eqref{z4} and \eqref{o80} it is obvious that the following three statements hold:
 \begin{align}\label{o81}
 \forall \varepsilon_0>0\  \exists K_1(\varepsilon_0),\ \forall k:&  N_k \le K_1(\varepsilon_0)\cdot L_k^{\beta_\theta-1+\varepsilon_0},\\
\forall \varepsilon_0>0\  \exists K_2(\varepsilon_0), \ \forall k:& M_k \le K_2(\varepsilon_0)\cdot L_k^{\beta_\theta-1+\varepsilon_0},\nonumber \\
\label{o83}
\forall \varepsilon_0>0\  \exists a_k\uparrow\infty , \ \forall k:& \ N_{a_k} \ge L_{a_k}^{\beta_\theta-1-\varepsilon_0}.
\end{align}
In particular, the last statements \eqref{o83} yields the definition of a sequence $(a_k)_{k\ge 1}$: the $k$th element of the sequence $a_k$ is defined as the $k$th index of $N$ for which there are enough $\lambda_j$ falling in the interval $N_{a_k}$.
We are ready to define the regularity of the sequence $(\lambda_j)_{j\ge 1}.$
%\begin{equation}\label{o81}
% \forall \varepsilon_0>0\  \exists K_1(\varepsilon_0),\ \forall k:  N_k \le K_1(\varepsilon_0)\cdot L_k^{\beta_\theta-1+\varepsilon_0},
%\end{equation}
%
%\begin{equation}\label{o82}
%\forall \varepsilon_0>0\  \exists K_2(\varepsilon_0), \ \forall k:  M_k \le K_2(\varepsilon_0)\cdot L_k^{\beta_\theta-1+\varepsilon_0},
%\end{equation}
%\begin{equation}\label{o83}
%\forall \varepsilon_0>0\  \exists a_k\uparrow\infty , \ \forall k: \ N_{a_k} \ge L_{a_k}^{\beta_\theta-1-\varepsilon_0}
%\end{equation}
%
\begin{definition}\label{t220}
  We say that the sequence $\left(\lambda _j\right)_{j=1}^{\infty }$ is regular if
$\forall \varepsilon _0>0$ the sequence $\left(a_k-a_{k-1}\right)_{k\geq 2}$ is bounded. We denote
\begin{equation}\label{ajuli}A:=A(\varepsilon)=\max_{k\ge 1}\left(a_k-a_{k-1}\right)\end{equation}
\end{definition}

 In what follows we always assume that $\left(\lambda_j\right)_{j\ge 1}$ is regular.

\subsection{The Multifractal Spectrum of TCP CUBIC}
We present here our main results of the multifractal spectrum of TCP CUBIC.

 Let us define the following two regions in the first quadrant in the $\alpha, \beta$ plane:
\begin{align*}%\label{z5}
 R_1&:=\left\{
 (\alpha,\beta):
 1 \leq \beta \leq 2 \mbox{ and }
 0
\leq \alpha \leq \frac{1}{\beta-2/3}
 \right\},\\
\label{z6}
R_2&:=\left\{
 (\alpha,\beta):
 1 \leq \beta \leq 2 \mbox{ and }
 \frac{1}{\beta-2/3} \leq
 \alpha
  \leq
  1+\frac{1}{\beta-2/3}
 \right\}.
\end{align*}

\clrgreen{ Furthermore, we partition $R_2$ into the lower and upper part $R_{2}^{\ell }$ and $R_{2}^{u}$.
 \begin{equation*}%\label{z94}
 R^{\ell }_2:=\left\{
 (\alpha,\beta)\!:\!
 \beta \in[1,2],\
\alpha
  \in \big[\frac{1}{\beta\!-2/3},
  \frac{\beta}{\beta\!-2/3}\big]
 \right\}.
 \end{equation*}
  and
   \begin{equation*}%\label{z94}
 R^{u}_2:=\left\{
 (\alpha,\beta)\!:\!
 \beta \in[1,2],\
\alpha
  \in \big[
  \frac{\beta}{\beta\!-2/3},
  1+\frac{1}{\beta\!-2/3}\big]
 \right\}.
 \end{equation*}}

\clrgreen{\begin{theorem}\label{z30}
Assume that $\left(\lambda_j\right)_{j=1}^{\infty}$ is regular. Then we have the following estimates on the multifractal spectrum of  TCP CUBIC for the increments ( $f^{(C)}(\alpha)$) and for the oscillations  $f_g^{\mathrm O(C)}(\alpha)$:
%\begin{description}
%  \item[(a)] For  $(\beta,\alpha)\in R_1$ we have
%$f^{(C)}_g(\alpha)=\alpha(\beta_3-2/3)$,
%  \item[(b)] For $(\beta,\alpha)\in R_2$ we have
%$f^{(C)}_g(\alpha) \leq 1+\frac{1}{\beta_3-2/3}-\alpha$,
%\end{description}
\begin{description}
  \item[(a)]
$f^{(C)}_g(\alpha)=\alpha(\beta_3-2/3)$ on   $R_1$,
\item[(b)]
$f^{(C)}_g(\alpha) \leq 1+\frac{1}{\beta_3-2/3}-\alpha$ on $R^\ell _2$.
\item[(c)]
$f_g^{\mathrm O(C)}(\alpha) \leq 1+\frac{1}{\beta_3-2/3}-\alpha$ on $R_2$.
\end{description}
where $\beta_3$ is as in \eqref{z4}.
\end{theorem}
}
The proof of this result follows from the proof of a more general result, Theorem~\ref{o79} presented in \secref{sec:generalization}.

\section{Comparison of the Multifractal Spectra of TCP CUBIC and TCP Reno}
\label{sec:comparison}

\begin{figure}%[H]
  \centering
  \includegraphics[width=7.5cm]{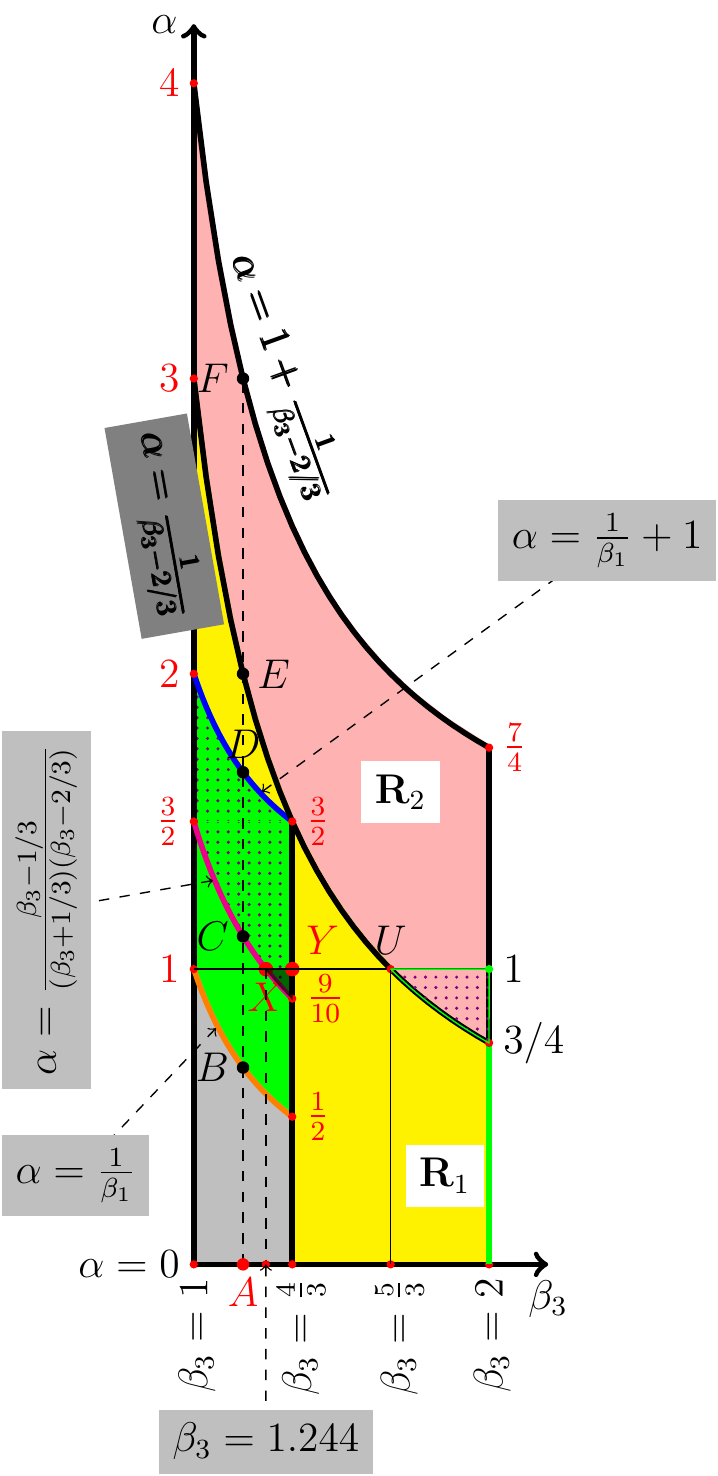}\\
  \caption{Parameter space of TCP Reno and TCP CUBIC for $\beta_1=3\left(\beta_3-\frac{2}{3}\right)$.}\label{z12}
\end{figure}
In this Section we compare our results, i.e., the multifractal spectrum of TCP CUBIC in Theorem \ref{z30} and the multifractal spectrum of TCP Reno obtained by L\'evy-V\'ehel, Rams \cite[Theorem III.4]{Rams2012} that we cite here for the readers convenience.
Recall the definition of $\beta_\theta$ from equation \eqref{z4}.

 \begin{theorem}[L\'evy-V\'ehel, Rams]\label{z29}
  Assume that $\left(\lambda_j\right)_{j\ge 1}$ is regular and
\begin{equation}\label{z10}
  \sum\limits_{j=1}^{\infty }\frac{1}{\lambda_j}<\infty
\end{equation}
The large deviation multifractal spectra for TCP Reno for $\beta_1\in(1,2)$ is
\begin{equation*}\label{o47}
  f^{(R)}_g(\alpha )=
\left\{
  \begin{array}{ll}
 \beta_1 \alpha    , & \hbox{if $\alpha \in [0,1/\beta_1 ]$;} \\
    1+1/\beta_1 -\alpha  & \hbox{if $\alpha \in [1/\beta_1 ,1+1/\beta_1 ]$;} \\
    -\infty  & \hbox{otherwise.}
  \end{array}
\right.
\end{equation*}
\end{theorem}
It is elementary to see from the definition of $\beta_\theta$ in \eqref{z4} that
\begin{equation}\label{z8}
  \beta_1=3\left(\beta_3-\frac{2}{3}\right).
\end{equation}
Using this identity, we obtain that under condition \eqref{z10}
$
  \beta_3\in\left(1,\frac{4}{3}\right).
$
Now we fix an arbitrary $\beta_3\in\left(1,\frac{4}{3}\right)$ and  vary $\alpha$. This implies that the region considered by L\'evy-V\'ehel and Rams in Theorem \ref{z29} is contained in the region $R_1$, see also Figure~\ref{z12}. We remind the reader that in region $R_1$ we have a complete result (i.e., matching upper ad lower bounds, see Theorem \ref{z30}).  Based on the given parameter sequence $\left(\lambda_j\right)_{j=1}^{\infty }$ of our model,
$\beta_\theta$ for all $\theta \geq 1$ are determined by \eqref{z4}.
For these $\beta_1$ and $\beta_3$ (c.f. \eqref{z8}) we compare the  multifractal spectra of TCP Reno
($f_{g}^{(R)}(\alpha)$) and TCP CUBIC ($f_{g}^{(C)}(\alpha)$).

This comparison means that we move $\alpha$ upwards  on the dashed vertical line in Figure \ref{z12}, starting from $(\beta_3,0)$ (point $A$) all the way up to the point $E$ which is the intersection between the dashed vertical line $\beta_3=\mathrm{const}$ and the upper boundary of region $R_1$.
The behavior of the large deviation multifractal spectra on this dashed line is shown in Figure \ref{z13}. We obtain the following corollary:

\begin{corollary}\label{z26} When comparing the multifractal spectrum of TCP RENO, $f_{g}^{(R)}(\alpha)$ and that of TCP CUBIC $f_g^{(C)}(\alpha)$, we have the following inequalities:
\begin{description}
  \item[(a)] For  $1 \leq \beta_3 \leq 1.244$ and for  $\alpha \leq 1$ we have $f_{g}^{(C)}(\alpha)<f_{g}^{(R)}(\alpha)$.
  \item[(b)] For $\beta_3\in\left(1.244,\frac{4}{3}\right)$ and for
   $\alpha<\frac{\beta_3-1/3}{(\beta_3+1/3)(\beta_3-2/3)}$ we again have
   $f_{g}^{(C)}(\alpha)<f_{g}^{(R)}(\alpha)$. In particular this happens for all $\beta<4/3$ when $\alpha<\frac{9}{10}$.
   \item[(c)] If $\beta_3\in\left(1.244,\frac{4}{3}\right)$ then for
   $\alpha>\frac{\beta_3-1/3}{(\beta_3+1/3)(\beta_3-2/3)}$ we have
   $f_{g}^{(C)}(\alpha) \geq f_{g}^{(R)}(\alpha)$.
\end{description}
\end{corollary}
In the rest of this section  we use the notation of Section \ref{z28}. Recall that $J_{\alpha, \ell, \varepsilon}$ stands for those level-$\ell$ diadic intervals on which the increment of $Z$ on is approximately, $|2^{\ell}|^\alpha$. Further, $N_\ell^\varepsilon(\alpha)=\# J_{\alpha, \ell, \varepsilon}$ is the number of these diadic intervals and finally $f_g(\alpha)\sim \log N_\ell^\varepsilon(\alpha)$.
 Note that the smaller the $\alpha$, the larger the increment and hence a larger multifractal spectrum function $f$ for the case $\alpha<1$  is of great importance to traffic analysis point of view since in this case the traffic is bursty.
%\footnote{Burst is a group of consecutive packets with shorter interpacket gaps than packets arriving before or after the burst of packets.}.
The multifractal spectra of TCP Reno and TCP CUBIC shows that both TCP versions generate bursty traffic, however, we see in Corollary \ref{z26} that for almost all cases $f_{g}^{(C)}(\alpha)<f_{g}^{(R)}(\alpha)$, i.e., TCP Reno is more bursty than TCP CUBIC. It can happen only for $\beta_3\in\left(1.244,\frac{4}{3}\right)$ and $\frac{9}{10}<\alpha$ that $f_{g}^{R}(\alpha)<f_{g}^{C}(\alpha)$. This exceptional region is the small black triangle with one side aligned  with the $\beta_3=4/3$ line with right upper vertex $Y$ in Figure \ref{z12}. However, in this special case the contribution to burstiness is not large since it comes only from $\alpha>9/10$, thus the increments of $Z$ in this region are smaller than in the case of small $\alpha$'s. In other words, the increments with small $\alpha$ values dominate the traffic burstiness. As a general observation we can conclude that the traffic of TCP CUBIC is less bursty than the traffic of TCP Reno.

In the above comparison we discussed the behavior of traffic for those $(\lambda_j)_{j\ge 1}$ sequences for which both the TCP Reno solution and the TCP CUBIC exists, i.e., $\beta_3\in\left(1,\frac{4}{3}\right)$. However, for $\beta_3>\frac{4}{3}$, we have solution for TCP CUBIC, see Theorem~\ref{z30} and Figure~\ref{z12}. So, we have the following result.

\begin{corollary}%\label{z26b}
\begin{description}\
  \item[(a)] For  $\frac{4}{3} < \beta_3 \leq \frac{5}{3}$ and for  $\alpha \leq 1$ we have $f^{(C)}_g(\alpha)=\alpha(\beta_3-2/3)$.
  \item[(b)] For $\frac{5}{3}< \beta_3 \leq 2$ and for  $\alpha \leq \frac{1}{\beta_3-2/3}$ we again have
   $f^{(C)}_g(\alpha)=\alpha(\beta_3-2/3)$.
   \item[(c)]  For $\frac{5}{3}< \beta_3 \leq 2$ and for  $\frac{1}{\beta_3-2/3}$ < $\alpha \leq 1$ we have
   $f^{(C)}_g(\alpha) \leq 1+\frac{1}{\beta_3-2/3}-\alpha$.
\end{description}
\end{corollary}

\begin{figure}%[H]
  \centering
  \includegraphics[width=8cm]{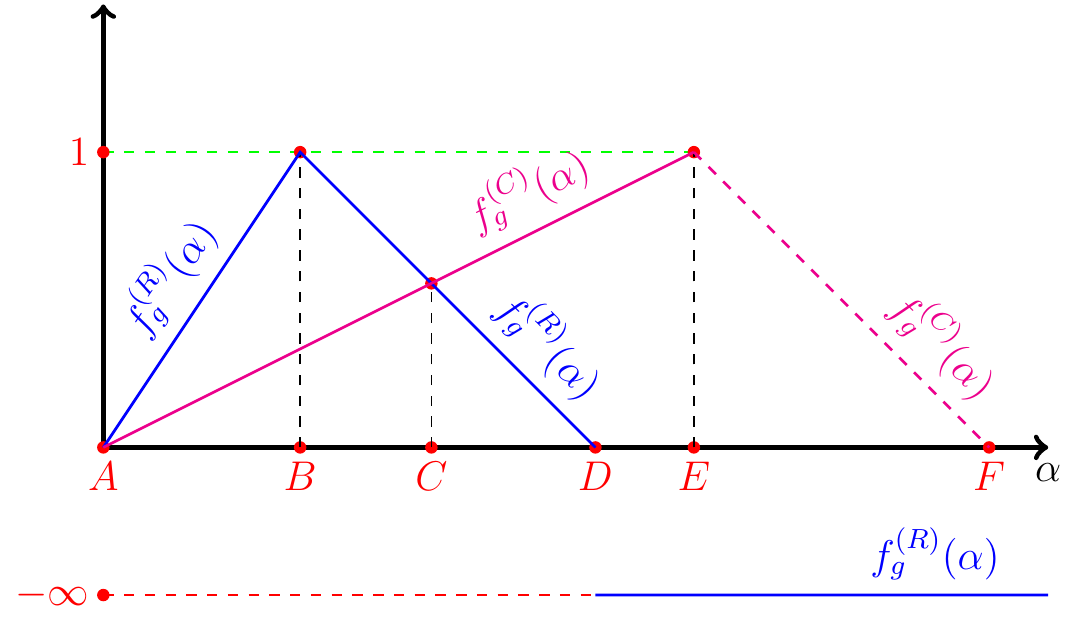}\\
  \caption{Multifractal spectra of TCP Reno and TCP CUBIC. This happens on the vertical dashed line segment $A-E$ in Figure \ref{z12}.}\label{z13}
\end{figure}

\section{Generalization of the TCP CUBIC Process}
\label{sec:generalization}

We shall prove the results stated in the previous section for a more general family of random processes. This general family includes not only both the TCP CUBIC and TCP Reno as special cases but many other stochastic processes which are infinite sums of random functions. The main point of the generalization is that we replace the very specific family, $\left\{g_x(t)\right\}_{x>0}$ defined in \eqref{o84}, with a much more general family of functions. This generalization is carried out based on the self-affine property \eqref{o85} of  $\left\{g_x(t)\right\}_{x>0}$ (cf.\eqref{0123}).

\subsection{Heuristic description of the generalization with an example}\label{z22}

In the general case, we also consider the infinite sum
\begin{equation*}\label{z18}
  Z(t)=\sum\limits_{j=1}^{\infty }Z_j(t),
\end{equation*}
 where $Z_j(t)$ is defined in a way which is similar to the case of TCP CUBIC model:
\begin{itemize}
  \item $Z_j(t)$ increases according to a deterministic rule in between two consecutive random points of losses.
  \item The random points of losses of $Z_j(t)$ are chosen according to a Poisson process of intensity $\lambda_j$ with $1=\lambda_1<\lambda_2<\dots$
  \item The deterministic rule of growth between the consecutive points of losses are governed by a self-affine family of functions like the one in \eqref{o85} with the exponent $3$ in \eqref{o85} replaced by a general  $\theta \geq 1$.
\end{itemize}
We remark that the $\theta=1$ case is essentially settled by L\'evy-V\'ehel, Rams \cite[Theorem III.4]{Rams2012}.
\begin{comment}
More precisely,
in Subsection \ref{z14} in the definition of $Z_i(t)$ there were four parameters:
\begin{itemize}
\item a characteristic exponent of the process $\theta$ which was $\theta=3$ in the case of the CUBIC TCP and $\theta=1$ in the case of Reno TCP. In the general case $\theta \geq 1$ is an arbitrary real number.
\item In the general case, like in the case of CUBIC TCP, we are given an increasing sequence of the intensities $\left\{\lambda_i\right\}_{i=1}^{\infty }$ with $\lambda_1=1$.
  \item For every $j \geq 1$ we are given the independent  $\mathrm{Poisson}(\lambda_i)$ processes $\left\{T_{k}^{(j)}\right\}_{k=1}^{\infty }$ and the corresponding interevent process $\tau_{k}^{(j)}:=T_{k}^{(j)}-T_{k-1}^{(j)}$.

  \item The major step in the generalization is as follows:
   We replace
   the self-affine family of deterministic functions $\left\{g_x(t)\right\}_{x>0}$ which was determined in the case of CUBIC TCP by \eqref{o84}, is replaced with a much more general self-affine family of functions defined below.
\end{itemize}
\end{comment}

To highlight the meaning of the abstract definition of $\left\{g_x(t)\right\}_{x>0}$ given below in Section  \ref{z31}, as an intermediate step, first we give an example which is included in the general case.

\begin{example}\label{z15}
Let $\left\{g_x(t)\right\}_{x>0}$ be defined as follows:
  \begin{equation}\label{t123}
    g_x(t)=x \cdot g_1\left(\frac{t}{x^{1/\theta}}\right),
    \quad x>0,t>0,
  \end{equation}
where $g_1(t)$ is  an arbitrary polynomial satisfying:
\begin{description}
  \item[(a)] The order of $g_1$ is $\theta$,
  \item[(b)] $g'_1(t) \geq 0$ for every $t\in\mathbb{R}^+$ and $g'_1(0)>0$,
  \item[(c)] $g_1(0)\in(0,1)$.
\end{description}
\end{example}
Example~\ref{z15} covers both TCP CUBIC and TCP Reno. Namely, we get the TCP CUBIC model with the choice of $\theta=3$ and
 \begin{equation}\label{t147}
      g_x(t)=x+C\left(t-\sqrt[3]{\frac{b}{C}} \cdot x\right)^3
      \mbox{ with }
 b=0.7,\ C=0.4.
    \end{equation}
 Similarly, the TCP Reno is included in Example~\ref{z15} with $\theta=1$ and $g_x(t):=x/\mu+t$ for a constant $\mu>1$.

 \subsection{The definition of $Z(t)$ in the general case}\label{z31}
 The most general definition of the family $\left\{g_x(t)\right\}_{x>0}$  given below differs from the one in Example \ref{z15} in the following way: We preserve the self-affine property by assuming  (A1) below.  Although we no longer require that $\theta$ is an integer, we would still like to preserve some properties of order $\theta$ polynomial $g_1(t)$ in Example \ref{z15}.  This is why we assume (A2) and (A3) below.
\begin{definition}\label{z16}
  For every $x>0$, $g_x:(0,\infty )\to(0,\infty )$  such that $(x,t)\mapsto g_x(t)$ is a $C^\infty $
  function satisfying
  the following assumptions:
There exists a $\theta \geq 1$ exponent such that

\begin{description}
  \item[(A1)] \emph{Self-affine property:}
  For every $0<r,t<\infty $ we have
  \begin{equation}\label{0123}
    \frac{1}{r^\theta}g_x(t)=g_{x/r^\theta}
    \left(\frac{t}
    {r}\right)
  \end{equation}
  The properties (A2) and (A3)guarantee  that $g_1(t)$ behaves similar to the polynomial in Example \ref{z15}:
  \item[(A2)] \emph{Growth properties:}
  \begin{description}
  \item[(A2a)]
  The derivative of $g^{1/\theta}_1(t)$ is a bounded function on $[0,\infty )$. That is,
 \begin{equation*}\label{t142}
\exists \psi>0,\ \dfrac{d}{dt}\left(g_1(t)^{1/\theta}\right)<\psi, \quad \forall t \geq 0.
 \end{equation*}
\item[(A2b)] There exits $c_1>0$ such that
\begin{equation}\label{t143}
  g_1(t) \geq c_1 t^\theta.
\end{equation}
\end{description}
\item[(A3)] \emph{Regularity property:}
We assume that
$g'_1(t)$ has finitely many zeros, $g'_1(t) \geq 0$ that is $g_1(t)$ is increasing and
\begin{equation*}\label{z17}
  g'_1(0)>0 \mbox{ and }
  \eta:=g_1(0)\in(0,1).
  \end{equation*}
\item[(A4)] The sequence of the intensities $\left(\lambda_j\right)_{j=0}^{\infty }$ of the independent Poisson point processes $\left(T_{k}^{(j)}\right)$ (defined  in Section \ref{z14}) satisfy:
    \begin{enumerate}
     \item $1=\lambda_1<\lambda_2< \cdots <\lambda_n< \cdots $
  \item $\left(\lambda_j\right)_{i=1}^{\infty }$ is regular in the sense of Definition \ref{t220})
  \item We assume that
  \begin{equation}\label{o59}
\sum\limits_{j=1}^{\infty }\frac{\log \lambda_j}{\lambda _{j}^{\theta}}<\infty .
\end{equation}

    \end{enumerate}
\end{description}
\end{definition}
Note that \eqref{0123}
means that for $0<a$
we have
\begin{equation*}\label{z25}
  \left(
  \begin{array}{cc}
    a & 0 \\
    0 & a^\theta \\
  \end{array}
\right) \cdot
\mathrm{graph}(g_x)=\mathrm{graph}(g_{a^\theta x}).
\end{equation*}
This is why we call the family $\left\{g_x(t)\right\}_{x>0}$ self-affine.
The definition of the random function $Z_j(t)$ in the general case is the same as in Section \ref{z14} with the only modification that we use in \eqref{o58} the previously defined more general version of $\left\{g_x(t)\right\}_{x>0}$. That is, $Z_j(t)$ is defined in a right-continuous way:
\begin{equation*}\label{z19}
  \large{Z_j(t):=g_{Z_j(T^{(j)-}_{k-1})}\left(t-T^{(j)}_{k-1},\right)}.
\end{equation*}
Observe that by the self-affine property of $g_x(t)$ we have the distributional identity
\begin{equation}\label{t171}
  \frac{1}{\lambda_{j}^{\theta}}Z_1(\lambda_j t)
  \law Z_j(t),
\end{equation}
This completes the definition of  $Z(t)=\sum\limits_{j=1}^{\infty }Z_j(t)$. See Corollary~\ref{t226} that $Z(t)<\infty$,  $\forall t\in[0,1]$.

\subsection{Our result in the general settings}
From now on we always write $f_g(\alpha):=f_{g}^{(\theta)}(\alpha)$ for the Large Deviation Multifractal spectrum of $Z(t)$ (see Definition \ref{t205}).
Observe that it follows from the definition of  $f_g(\alpha)$ and from \eqref{t171} that $f_g(\alpha)$ remains the same if we change from the sequence of intensities $\left(\lambda_j\right)_{j\ge 1}$ to $ \left(\mathrm{const} \cdot\lambda_j\right)_{j\ge 1}$. So, without loss of generality we may assume that $\lambda_1=1$.

First we define the regions of the $\alpha,\beta$ plane
 \begin{equation}\label{z20}
     R_1:=\left\{
 (\alpha,\beta)\!:\!
 \beta\in[1,2],\
     \alpha\in \big[0, \frac{1}{\beta-(1-1/\theta)}\big]
 \right\},
 \end{equation}
and
\begin{equation*}\label{z21}
 R_2:=\left\{
 (\alpha,\beta)\!:\!
 \beta \in[1,2],\
\alpha
  \in \big[\frac{1}{\beta\!-\!(1\!-\!1/\theta)},
  1+\frac{1}{\beta\!-\!(1\!-\!1/\theta)}\big]
 \right\}.
 \end{equation*}
\clrgreen{ Furthermore, we partition $R_2$ into the lower and upper part $R_{2}^{\ell }$ and $R_{2}^{u}$.
 \begin{equation*}\label{z94}
 R^{\ell }_2:=\left\{
 (\alpha,\beta)\!:\!
 \beta \in[1,2],\
\alpha
  \in \big[\frac{1}{\beta\!-\!(1\!-\!1/\theta)},
  \frac{\beta}{\beta\!-\!(1\!-\!1/\theta)}\big]
 \right\}.
 \end{equation*}
  and
   \begin{equation*}\label{z94}
 R^{u}_2:=\left\{
 (\alpha,\beta)\!:\!
 \beta \in[1,2],\
\alpha
  \in \big[
  \frac{\beta}{\beta\!-\!(1\!-\!1/\theta)}\big],
  1+\frac{1}{\beta\!-\!(1\!-\!1/\theta)}
 \right\}.
 \end{equation*}}

 See  Figure \ref{o60}.

   \begin{figure}[H]
  \begin{center}
  \includegraphics[height=8cm]{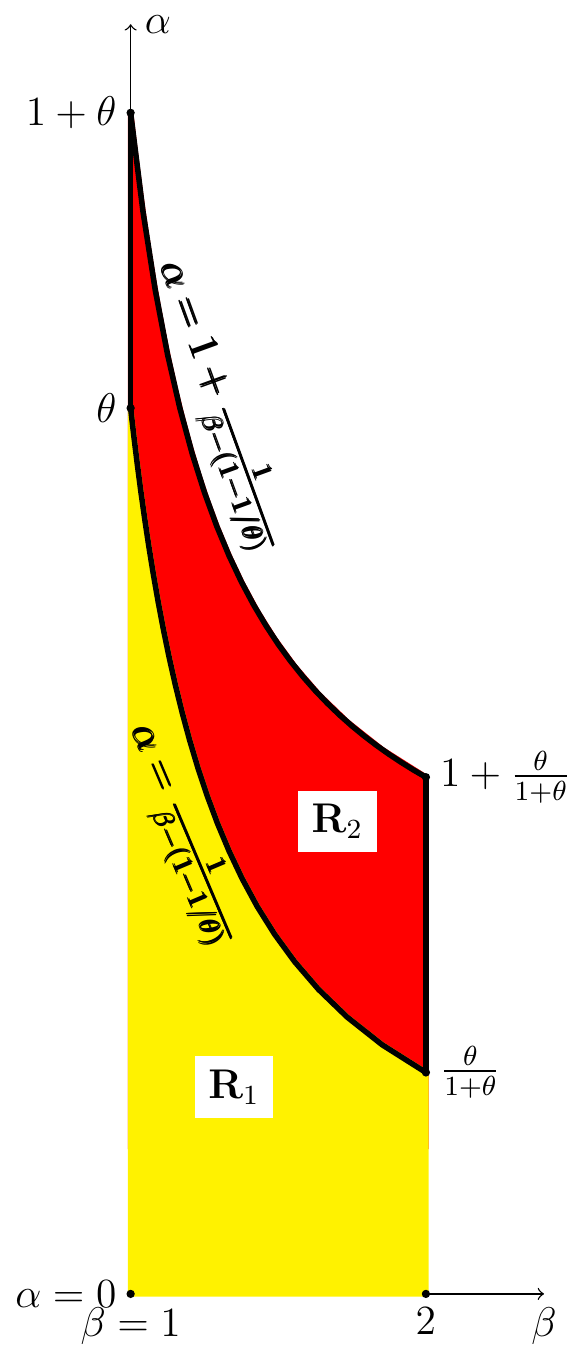}
   \includegraphics[height=8cm]{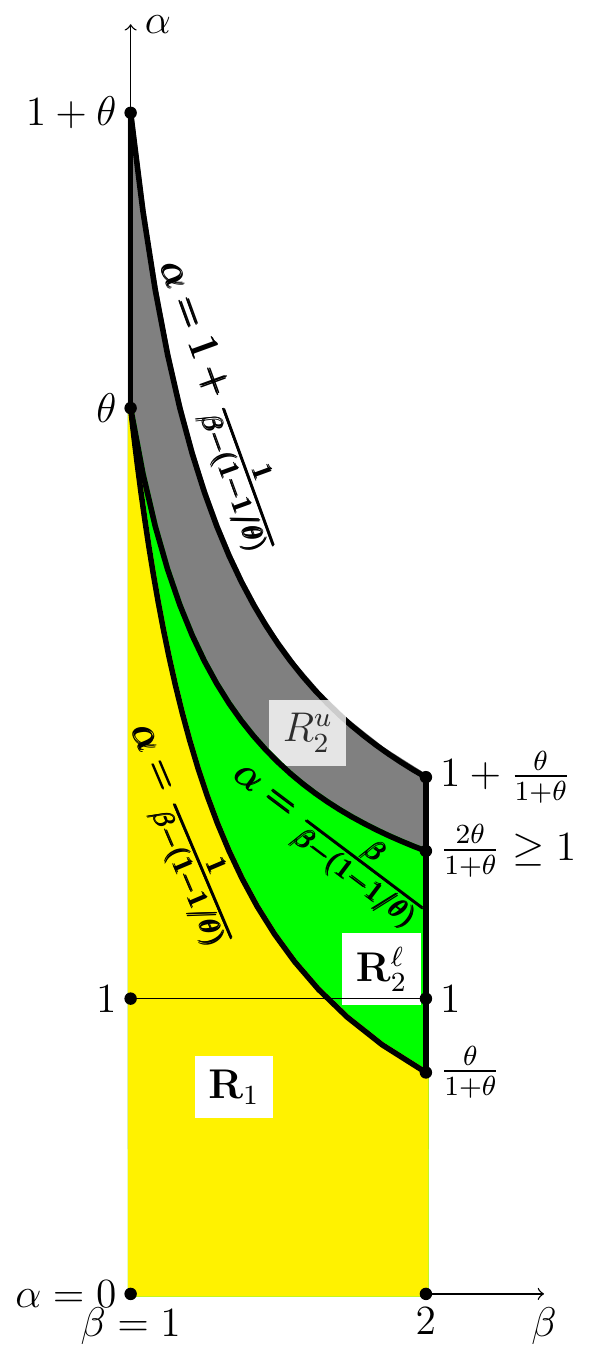}\\
  \caption{Here $\beta=\beta_\theta$ defined in \eqref{z4}. }
  \label{o60}
  \end{center}
\end{figure}
As a generalization of Theorem \ref{o79} we state:
 \clrgreen{\begin{theorem}\label{o79}
Let $f_g(\alpha)$ be the large deviation multifractal spectrum for the increments and $f_g^{\mathrm O}$ for the oscillations
 of the random function $Z(t)$ defined in Section \ref{z22}. We assume that the sequence of intensities $\left(\lambda_j\right)_{j=1}^{\infty }$ satisfies the assumption {\bf(A4)} in Definition \ref{z16}. Then we have
%\begin{description}
%  \item[(a)] If $(\beta,\alpha)\in R_1$ then
%$f_g(\alpha)=\alpha(\beta-(1-1/\theta))$,
%  \item[(b)] If $(\beta,\alpha)\in R_2$ then:
%$f_g(\alpha) \leq 1+\frac{1}{\beta-(1-1/\theta)}-\alpha$.
%\end{description}
\begin{description}
  \item[(a)] $f_g(\alpha)=\alpha(\beta_{\theta}-(1-1/\theta))$ on $R_1$,
  \item[(b)]
$f_g(\alpha) \leq 1+\frac{1}{\beta_{\theta}-(1-1/\theta)}-\alpha$ on $R^\ell _2$,
 \item[(c)]
$f_g^{\mathrm O}
(\alpha) \leq 1+\frac{1}{\beta_{\theta}-(1-1/\theta)}-\alpha$ on $R _2$.
\end{description}
 \end{theorem}}
The rest of the paper is devoted to the proof of Theorem \ref{o79}. First we give a heuristic argument to show why it is natural to expect that at least the first part of this theorem holds.

\subsection{Heuristic proof for Part (a) of Theorem \ref{o79}}

\clrgreen{Fix a $\beta\in[1,2]$ and choose  $\alpha$ such that
\begin{equation}\label{z49}
  \alpha\left(\beta-1+1/\theta\right)<1,
\end{equation}
that is we are on the region $R_1$. Fix  a small $h=2^{-\ell}>0$. We would like to compute the number of $2^{-\ell }$-mesh intervals on which the magnitude of the increment of $Z$ is approximately $h^\alpha$. To do so, we
consider the  $L$-block $\left(L^{k-1},L^k\right)$ which `corresponds to' $h$ and $\alpha$. That is we define $k$ such that
\begin{equation}\label{z101}
  L^{k-1}<h^{-\alpha} \leq L^k.
\end{equation}
We say that an index $j$ is $\alpha$-good if
$$
\lambda_j^{\theta}\in(L^{k-1},L^k).
$$
We will see that, roughly speaking, a typical jump for the process $Z_j$ is of magnitude $\lambda_{j}^{-\theta}$ and these jumps happen roughly once in a time interval of length $1/\lambda_{j}$. We will prove that
\begin{itemize}
  \item if $\lambda_j$ is greater than the $\alpha$-good parameters, then although the process $Z_j(t)$ jumps frequently, these jumps are too small to influence the outcome as far as we count increments of magnitude $h^\alpha$.
  \item On the other hand, if $\lambda_j$ is smaller than the $\alpha$-good parameters,  then
  although the jumps of the process $Z_j(t)$ are big but these jumps happen rarely, so their effect is not significant for counting the jumps of  magnitude $h^\alpha$.
\end{itemize}
So, we can focus on the $\alpha$-good indices $j$. Observe that
\begin{itemize}
  \item Each process $Z_j$ with $\alpha$-good $j$ jumps approximately $\lambda_j\sim h^{-\alpha/\theta}$ times
  \item There are approximately $h^{-\alpha(\beta_{\theta}-1)}$ good indices $j$ for $\alpha$. Namely,
  By the definition of the Blumetal Getoor index $\beta_\theta$ and the regularity of $\left\{\lambda_j\right\}_{j=1}^{\infty }$  (see Section \ref{z102}) we have approximately
  $L^{k(\beta_\theta-1)}$ indices $j$ such that $\lambda_j^{-\theta}\in \left(L^{k-1},L^k\right)$. That is the number of good indices is approximately $L^{k(\beta_\theta-1)}$. However, by \eqref{z101} $L^k\sim h^{-\alpha}$. That is we have approximately $h^{-\alpha(\beta_{\theta}-1)}$ good indices.
\end{itemize}
So, the
 combined effect of these two points above suggests that there should be approximately
 \begin{equation}\label{z50}
   h^{-\alpha/\theta} \cdot h^{-\alpha(\beta_\theta-1)}=h^{-\alpha(\beta_\theta-1+1/\theta)}
 \end{equation}
$2^{-\ell }$-mesh intervals on which the magnitude of the  increments is $h^\alpha$. Here we
used that for different indices $j$ which are $\alpha$-good the majority of the corresponding mesh-intervals are different and well-separated from each other. This follows from the assumption \eqref{z49}. Observe that
 by \eqref{t212} and by $h=2^{-\ell }$,
 formula \eqref{z50} is actually part (a) of Theorem \ref{o79}.}

\section{Stationary measure of the Markov Chain associated to $Z_j$}
\label{sec:stationarity}
\cred{Now we turn to the technical details of the proofs.}
 Fix a $j\ge 1$. It is easy to see that $Z_j(t)$ is not a continuous time Markov process if $\theta>1$, that is, $g$ is not an affine function. Therefore we consider $\Phi_{k}^{(j)}$, the values of $Z_j$ just before the the $k$-th loss (see Figure \ref{z2}). Then for every $j\ge 1$, $\Phi^{j}:=\left(\Phi_{k}^{(j)}\right)_{k=1}^{\infty }$
 is a discrete time continuous state space  Markov chain. As  the  Meyn, Tweeedie book \cite{meyn2009markov} is a major reference book of this field  we use its terminology in this paper.
  Due to the self-affine property, it is enough to study the Markov chain $\Phi$ which corresponds  to a reference process $Z_0$ that is defined exactly as $Z_j$, for $\lambda:=\lambda_1:=1$.

\subsection{Geometric ergodicity of  $\Phi$}

In this section we  study the scalar  non-linear discrete time, continuous state space Markov chain model (following  the terminology in \cite{meyn2009markov}):
\begin{equation}\label{t130}
 \Phi _k=g_{\Phi_{k-1}}(T_k-T_{k-1}),
\end{equation}
where the function
    $g_x:[0,\infty )\to \mathbb{R}^+$ is
     defined for all $x> 0$ and
     satisfies  the assumptions {\bf(A1)-(A3)} of Definition  \ref{z16} and
      $\left(T_k\right)_{k\ge 1}$ is  $\mathrm{Poisson}(1)$ point process.

     Let $P$ be the probability kernel of the time-homogeneous Markov Chain $\Phi=\left(\Phi_k\right)_{k\ge 1}$. That is, \structure{for a set $A\subset \mathbb R^+$,}
        \begin{equation}\label{t132}
          P(x,A):=\mathbb{P}\left(\Phi_{k+1}\in A|\Phi_k=x\right).
          \end{equation}
We prove
\begin{theorem}\label{t131}\structure{For the Markov chain $\Phi$ described above}
\begin{description}     \item[(a)] there exists a unique stationary state $\pi$.
    \item[(b)] $\int V(x)d\pi(x)<\infty $ for
    \begin{equation}\label{t134}
      V(x):=\exp \left\{\structure{\tilde c} \cdot x^{1/\theta}\right\},
    \end{equation}
   where \structure{$\tilde c$} is positive constant defined in \eqref{t164}.
    \item[(c)] (Geometric Ergodicity) There exists constants $r>1$ and $R<\infty $ such that
\begin{equation}\label{t132}
  \sum\limits_{n}
  r^n\|P^n(x,\cdot)-\pi\|_V
   \leq
   R\cdot V(x),
\end{equation}
        where by definition for a non-negative function $f$, and the $f$-norm of a measure $\nu$ is defined as
        $
        \|\nu\|_f:=\sup\limits_{h:|h| \leq f}|\nu(h)|
        $, $|h| \leq f$ is meant pointwise and $\nu(f)$ is the integral of $f$ \structure{with respect to the measure $\nu$}.
  \end{description}
\end{theorem}
We prove Theorem \ref{t131} in Section \ref{t269}.
Part (b) of Theorem \ref{t131} immediately implies the following corollary:
\begin{corollary}\label{t165}
There exists a constant $K_3$ such that
  for every $k \geq 1$ we have
  \begin{equation*}\label{t166}
    \pi\left([k^\theta,(k+1)^\theta]\right) \leq K_3  \cdot
    \e{-\structure{\tilde c} \cdot k},
  \end{equation*}
  \structure{where $\tilde c$ is the constant in the exponent of $V(x)$, defined in \eqref{t164} below.}
\end{corollary}
In the rest of this section our aim is to prove Theorem~\ref{t131}. The assertions of Theorem~\ref{t131} follow from   two theorems (\cite[Theorem 14.0.1]{meyn2009markov},
\cite[Theorem 15.0.2]{meyn2009markov}) from Meyn, Tweedie book. The conditions of these two theorems are as follows:
\begin{itemize}
  \item[(i)]  Drift Condition (Proposition \ref{t136}),
  \item[(ii)] $\Phi$ is $\psi$-irreducible (Lemma \ref{t158}),
  \item[(iii)] $\Phi$ is recurrent (Proposition \ref{t156}),
  \item[(iv)] $\Phi$ is strongly aperiodic, (Lemma \ref{t155}),
\end{itemize}
for the terminology see \cite{meyn2009markov}.
In Section~\ref{z35}, as a preparation for the proof of Theorem~\ref{t131}, we study the self-affine family $\{g_x(t)\}$.
In Section~\ref{t269} we verify \structure{(i)} above. In Section~\ref{z36} we prove \structure{(ii)-(iv)} above and as a consequence of these we also prove Theorem~\ref{t131}. Finally in Section \ref{z37} we describe some properties of the density of the kernel.

\subsection{Properties of $g_x(t)$}\label{z35}

 \begin{figure}[H]
  \begin{center}
  \includegraphics[width=9cm]{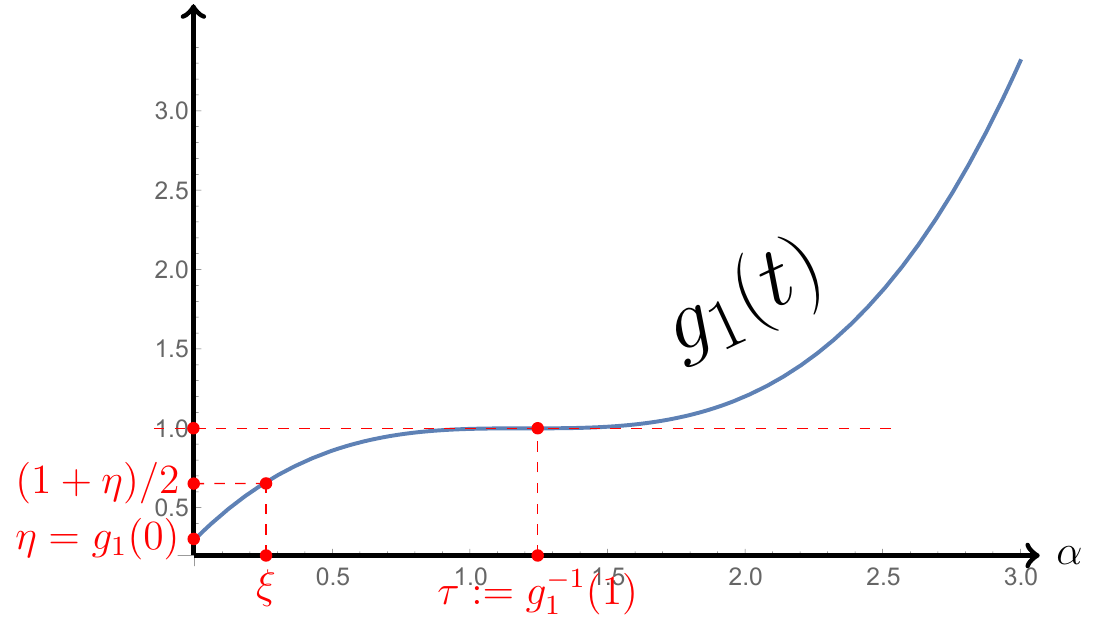}\\
  \caption{$g_1(t)$ in the CUBIC TCP case.}
  \label{z32}
  \end{center}
\end{figure}

Here we frequently use the following notation  (see Figure \ref{z32}):
\begin{equation*}\label{z33}
  \eta:=g_1(0) \mbox{ and } \xi:=g_{1}^{-1}\left(\frac{1+\eta}{2}\right) \mbox{ and }
  \tau:=g_{1}^{-1}(1),
\end{equation*}
where $g_{1}^{-1}\left(\cdot\right)$ is the inverse function of $g_1$.
We will use the following properties of $g_x(t)$:

\begin{remark}\label{t144}
\structure{Here we mention a few properties of the function $g_x(t)$.}
 \begin{enumerate}
 \item
   Substituting $r=x^{1/\theta}$  in {\bf(A1)} \structure{in Definition \ref{z16}} above, we obtain
     \begin{equation}\label{t123}
    g_x(t)=x \cdot g_1\left(\frac{t}{x^{1/\theta}}\right)
    \mbox{ and }
    g_{x}^{-1}(t)=x^{1/\theta}g_{1}^{-1}\left(\frac{t}{x}
    \right).  \end{equation}
  Hence
    \begin{equation}\label{z56}
  g'_x(t)=x^{1-1/\theta}g'_1\left(\frac{t}{x^{1/\theta}}\right).
    \end{equation}
Combining \structure{the first equation in} \eqref{t123} with \eqref{t143} we obtain
  \begin{equation}\label{z48}
    g_x(t) \geq c_1 \cdot t^\theta\qquad \forall x>0,\, t>0.
  \end{equation}
  \item \structure{With $\psi$ as in in Definition \ref{z16}, {\bf(A2a)},
the assumption {\bf(A2a)} implies that }
\begin{equation}\label{t167}
  g'_1(t)<\psi\theta (g_1(t))^{1-1/\theta}.
\end{equation}
Combining this inequality first with \eqref{z56} then with  \eqref{t123}  yields that for $x>0$
\begin{equation}\label{t238}
\structure{g_x'(t)} \leq \psi \cdot \theta \cdot g_x(t)^{1-1/\theta}.
\end{equation}
    \item Using (1) and (2) one can easily see that if $g_1(t)$ satisfies the assumptions {\bf(a)-(c)} of Example \ref{z15} then family
    $\left\{g_x(t)\right\}_{x>0}$ defined by \eqref{t123}  satisfies the assumptions {\bf(A1)-(A3)} of Definition \ref{z16}.

 \item It is easy to see that
$
   g_x\left(\tau  \cdot x^{1/\theta}\right)=x.
 $
 \item \emph{Forward accessibility of the associated control system. }
    Observe that
    \begin{equation*}\label{t126}
      g_x(0)=\eta\cdot x \mbox{ that is }
      g_x[0,\infty )=[\eta \cdot x,\infty ].
    \end{equation*}
    This implies that for the recursively defined functions
    $$g_x^{(n)}(t_1,  \dots ,t_n):=
    g_{g_x^{(n-1)}(t_1,  \dots ,t_{n-1})}(t_n)
    $$
    we have $g_{n}^{(x)}(\underbrace{0, \dots ,0}_n)=\eta^n \cdot x$. That is
    \begin{equation*}\label{t129}
      \bigcup_{n=1}^{\infty }
      g_{x}^{(n)}\left([0,\infty )^n\right)
      =
      \left(0,\infty \right).
    \end{equation*}
     We say that $g_x^{(n)}(t_1,  \dots ,t_n)$ is  \emph{the associated control system} of the Markov chain
     $\structure{\Phi_k}$ (defined in \eqref{t130})
driven by $g_x(t)$. By definition (see \cite[p. 141]{meyn2009markov}) this means that the associated control system is \emph{forward accessible}.
  \end{enumerate}

  \end{remark}

\structure{The next fact is simple but is important, so we list it separately.}
 \begin{fact}%\label{t157}
  \structure{Suppose that $\{g_x(t)\}_{x\ge 0}$ satisfies the assumptions in Definition \ref{z16}. Then}
 there exists $c_2>0$ such that
\begin{equation}\label{t145}
  g_1(t) \leq c_2t^\theta, \quad \mbox{ for }
  t \geq \xi.
\end{equation}
%  \begin{enumerate}
%    \item There exists $c_2>0$
%\begin{equation}\label{t145}
%  g_1(t) \leq c_2t^\theta, \quad \mbox{ for }
%  t \geq \xi.
%\end{equation}
%
%
%\end{enumerate}
\end{fact}

\begin{proof}
By the assumption \structure{\bf{(A2a)} in Definition \ref{z16}}, Lagrange Theorem implies that $$
g_{1}^{1/\theta}(t)\leq g_{1}^{1/\theta}(\xi) +
\psi \cdot (t-\xi), \quad
  t>\xi.
$$
\structure{This implies that \eqref{t145} holds}.
\end{proof}

\begin{fact}\label{fact2} We need the following properties of $g_{x}^{-1}$, the inverse function of $g_x$.
\begin{description}
  \item[(a)] There exists a constant $c_3>0$ such that
  \begin{equation}\label{t147}
    (g_{x}^{-1})'(t) \geq c_3
    t^{-(1-1/\theta) }, \quad \forall x>0,\forall  t>x\eta.
  \end{equation}
  \item[(b)] \structure{With $c_1$ as in assumption {\bf(A2b)} in Definition \ref{z16},}
  \begin{equation}\label{t146}
  g_{x}^{-1}(t) \leq \left(\frac{t}{c_1}\right)^{1/\theta}
  \forall  x>0, \forall
  t> x\eta.
\end{equation}
\end{description}
\end{fact}

\begin{proof} Part {\bf(a)}:
Using that $(g_x^{-1})'(t)=1/g'_x(g_x^{-1}(t))$,
  it follows from \eqref{t123} that for every $x,t>0$ we have
 $$
  g'_{x}\left(g_{x}^{-1}(t)\right) =x^{1-1/\theta}
  g'_1\left(g_{1}^{-1}\left(t/x\right)\right)
$$
If we apply \eqref{t167} with  \structure{$t$ replaced by $g_{1}^{-1}\left(t/x\right)$} then we get
the assertion of the lemma with $c_3:=(\psi\theta)^{-1}$.

Part {\bf(b)}:
 Using assumption \structure{{\bf(A2b)} in Definition \ref{z16}} and second part of \eqref{t123} we obtain that \eqref{t146} holds.
\end{proof}

\subsection{The verification of the drift condition}\label{t269}

 In this section we prove Proposition~\ref{t136} that implies that the so-called Drift Condition holds. \structure{This is the first step towards proving Theorem \ref{t131}, see point (i) in the argument below Corollary \ref{t165}. }

We frequently use the drift operator $\Delta $ that is defined  for any measurable function $f:(0,\infty )\to(0,\infty )$ by
\begin{equation}\label{t133}
  \Delta f(x):=\int P(x,dy)f(y)-f(x)=
  \int\limits_{y=0}^{\infty }
  p(x,y)f(y)dy-f(x)
  ,\quad x>0,
\end{equation}
where
$p(x,y)$ is  the density of the kernel $P(x,dy)$. \structure{For the Markov Chain $\Phi_k$, by \eqref{t130}, elementary calculation using the density of the exponential distribution yields that this density kernel is given by}
\begin{equation}\label{t137}
  p(x,y):=
  \dfrac{\exp
  \left(
  -g_{x}^{-1}(y)
  \right)}
  {g'_x\left(g_{x}^{-1}(y)\right)}.
\end{equation}

Now we state and prove that the drift condition holds.
\begin{proposition}\label{t136} \structure{Recall the definition of $V(x)$ from \eqref{t134}.}
 There exists a $K>0$ such that
 \begin{equation*}\label{t135}
   \Delta V(x)<-\frac{V(x)}{2}+2 \cdot \ind_{(0,K]},\quad \forall x>0,
 \end{equation*}
 where $\ind_{(0,K]}$ is the indicator function of the interval $(0,K]$.
\end{proposition}

\begin{proof}[Proof of Proposition \ref{t136}]
Let $r:=\xi \cdot x^{1/\theta}$. \structure{Using the formula for the density kernel \eqref{t137} and the definition of the operator $\Delta$ in \eqref{t133}, we calculate}

\begin{multline}\label{t160}
 \Delta V(x)=\int\limits_{y=\eta \cdot x}^{\infty }
 V(y)\frac{\exp\left[-g_{x}^{-1}(y)\right]}
 {g'_x\left(g_{x}^{-1}(y)\right)}dy
   \\
 =\underbrace{\int\limits_{u=0}^{r}V
 \left(g_x(u)\right)\e{-u}du}_{I_1}
 +  \underbrace{\int\limits_{u=r}^{\infty }  V\left(g_x(u)\right)\e{-u}du}_{I_2}
 -V(x),
\end{multline}
where we applied the substitution $u=g_{x}^{-1}(y)$.
First we estimate
\begin{equation}\label{t161}
  I_1-V(x)  \leq
\int\limits_{u=0}^{r}\left(\e{\tilde c \cdot g_{x}^{1/\theta}(u)}
-
\e{\tilde c \cdot x^{1/\theta}}\right)\e{-u}du.
\end{equation}

Using \eqref{t123}
 and the fact that $t\mapsto g_x(t)$ is increasing we obtain that
$g_{x}^{1/\theta}(u) \leq \eta_{1}^{1/\theta} \cdot x^{1/\theta}$ holds for all $u\in\left(0,r\right]$, where
$\eta_1:=g_1(\xi)\in(\eta,1)$. From this inequality and \eqref{t161} follows that
\begin{equation}\label{t162}
  I_1-V(x) \leq
  \e{\tilde c \cdot x^{1/\theta}}
  \left[
  \e{\tilde c \cdot x^{1/\theta}\left(\eta_{1}^{1/\theta}-1\right)}
  -1\right]
 = V(x)\left[
  \e{\tilde c \cdot x^{1/\theta}\left(\eta_{1}^{1/\theta}-1\right)}
  -1\right]
  .
\end{equation}
 Note that \structure{since $\eta<1$,} the exponent on the right hand side is negative. Hence $I_1-V(x)<0$ and tends to $-\infty $ as $x\to\infty $. To estimate $I_2$ \structure{we
combine
\eqref{t123} again with \eqref{t145} to obtain that } $g_{x}^{1/\theta}(u) \leq c_{2}^{1/\theta}u$ if $u \geq r$. Hence
\begin{equation}\label{t163}
I_2 \leq \int\limits_{u=r}^{\infty }
\e{\tilde cg_x^{1/\theta}(u)-u}du
 \leq \int\limits_{u=r}^{\infty }
\e{(\tilde c \cdot c_2^{1/\theta}-1) \cdot u}du.
\end{equation}
Now we fix
\begin{equation}\label{t164}
  \tilde c:=\frac{1}{2 c_{2}^{1/\theta}}.
\end{equation}
Then by \eqref{t163},
we have
  $I_2  \leq 2$.

Combining  this and \eqref{t162}
we obtain from \eqref{t160} that
\begin{description}
  \item[(a)] $\Delta V(x)<2$ for all $x>0$.
  \item[(b)] We can choose a $K$ such that $\Delta V(x) \leq -\frac{1}{2}V(x)$ if $x \geq K$.
\end{description}
\end{proof}

\subsection{Recurrence, irreducibility and strong aperiodicity  of $\Phi$}\label{z36}

It is a most fundamental property of $\Phi$  that it is a  $\psi$-irreducible chain (for the definition see \cite{meyn2009markov}).

\begin{lemma}\label{t158}
  The chain $\Phi$ is $\psi$-irreducible.
\end{lemma}
\begin{proof}
Our chain  $\Phi$ is  a so-called scalar nonlinear model which clearly satisfies the conditions SNSS1-SNSS3 of
the book \cite{meyn2009markov}. Also the associated control system is forward accessible (see part (5) of Remark \ref{t144}). Then by  \cite[Proposition 7.1.2]{meyn2009markov} $\Phi$ is a so-called $T$-chain. Let
$$
L(x,A):=\mathbb{P}_x\left(
\Phi
\mbox
{ ever enters }A
\right)
$$
According to \cite[Theorem 6.0.1]{meyn2009markov} the chain $\Phi$ is a $\psi$-irreducible chain if
\begin{equation}\label{z38}
  L(x,O)>0,\quad \forall  x>0 \mbox{ and open set }O\subset(0,\infty ).
\end{equation}
However, it is obvious from the construction that \eqref{z38} holds in our case.
\end{proof}

The purpose of the next two technical lemmas are to
prove that the interval $(0,K]$ is a petite set (see \cite[Section 5.5]{meyn2009markov}) for all $K>0$. \structure{Roughly speaking, a set $A$ is a petite set if there is a function that serves as a \emph{uniform  lower bound} on the density of the transition kernel of the Markov Chain from an arbitrary $x \in A$ to the complement of $A$. We apply this for a $K$ which satisfies Proposition \ref{t136}. This yields a petite set $(0,K]$ out of $\Delta V<0$. Then by \cite[Theorem 8.0.2 (ii)]{meyn2009markov} he chain is recurrent.}

\begin{lemma}%\label{t149}
   \structure{Recall the constants $c_1$ from Definition \ref{z16}, {\bf (A2b)} and $c_3$ from \eqref{t147} from Fact \ref{fact2}}. For $t>0$ let us define the function
  \begin{equation*}\label{t150}
  \mathfrak{h}(t):=
 c_3 \exp\left(-\left(t/c_1\right)^{1/\theta}\right) \cdot
  t^{-(1-1/\theta)}.
    \end{equation*}
Then
\begin{equation*}\label{t151}
  \left|
  \frac{d}{dt}\exp\left(-g_{x}^{-1}(t)\right)
  \right|
  >
  \mathfrak{h}(t), \mbox{ if }
  t>\eta \cdot x.
\end{equation*}
\end{lemma}
\begin{proof}
The function $t\mapsto \exp\left(-g_{x}^{-1}(t)\right)$ monotone decreasing.
  By chain rule:
  $$
  -\frac{d}{dt}\exp\left(-g_{x}^{-1}(t)\right)
  =
\exp\left(-g_{x}^{-1}(t)\right)  \cdot
\frac{d}{dt}
 g_{x}^{-1}(t)
  $$
  We estimate  the first term from below by the inequality
   $ g_{x}^{-1}(t) \leq \left(t/c_1\right)^{1/\theta}$ for $x>0$ and $t>x \cdot \eta$
  and second term by\eqref{t147}. This yields the proof.
\end{proof}

In the rest of this section we use the terminology of the book \cite{meyn2009markov}.
For an arbitrary $K>0$ we define the measure $\nu^{(K)}$ supported on  $[K,\infty )$ by
\begin{equation}\label{t152}
  \nu^{(K)}([\alpha,\beta]):=
  \int\limits_{t=\alpha}^{\beta }\mathfrak{h}(t)dt,
\end{equation}
where $K \leq \alpha<\beta$.

\begin{lemma}\label{t153}
  For every $K>0$ the set $(0,K]$ is a small set (see Section \cite[Section 5.2]{meyn2009markov} for the definition.)
  and then by \cite[Proposition 5.5.3]{meyn2009markov} $(0,K]$ is also a petite set for all $K>0$.
\end{lemma}
\begin{proof} Let $K<\alpha<\beta$.
 Then for an $x\in (0,K]$ we have
  \begin{multline*}
  P(x,[\alpha,\beta])=
  \int\limits_{t=\alpha}^{\beta}
  p(x,t)dt=
  \int\limits_{t=\alpha}^{\beta}
  \dfrac{\exp
  \left(
  -g_{x}^{-1}(t)
  \right)}
  {g'_x\left(g_{x}^{-1}(t)\right)}dt   \\
=
  \exp\left[-g_{x}^{-1}(\alpha)\right]
  -
  \exp\left[-g_{x}^{-1}(\beta)\right] \geq \int\limits_{\alpha}^{\beta}\mathfrak{h}(t)dt.
  \end{multline*}
  That is for every $x\in(0,K]$ we have
 $
  P(x,[\alpha,\beta]) \geq \nu^{(K)}([\alpha,\beta]).
 $
 This implies that the interval $(0,K]$ is a small set. Then $(0,K]$ is also a petite set by \cite[Proposition 5.5.3]{meyn2009markov}.

\end{proof}

\begin{corollary}\label{t156}
The chain $\Phi$ is recurrent.
\end{corollary}
\begin{proof}
   Using Lemma \ref{t153} and Proposition \ref{t136} the assertion immediately follows from
   \cite[Theorem 8.0.2 (ii)]{meyn2009markov}.
\end{proof}

\begin{lemma}\label{t155}
  The chain $\Phi:=\left\{\Phi_n\right\}$ is strongly aperiodic (for the terminology see \cite[p. 114]{meyn2009markov})
\end{lemma}
\begin{proof}
  Let $A:=[\eta,1]$ and let $\nu:=\nu^{(\eta)}$ \structure{as defined in \eqref{t152}.}
  Then $\nu(A)>0$ and for every
  $x\in A$ we have $P(x,B) \geq \nu(B)$ holds for every $B\subset(0,\infty )$.
\end{proof}

 Now we are ready to prove the main result of the section.
\begin{proof}[Proof of Theorem \ref{t131}]\label{z103}
Part {\bf(a)} and {\bf(b)}:
We know that $\Phi$ is a $\psi$-irreducible and aperiodic chain (see Lemmas \ref{t158}, \ref{t155}). We have also verified that for any $K$ the set
$(0,K]$ is a petite set. Hence by Proposition \ref{t136},
the conditions  of (iii) of \cite[Theorem 15.0.2]{meyn2009markov} hold.
  It follows from this theorem
that $\Phi$ is positive recurrent with unique stationary measure  $\pi$ satisfying \eqref{t132}.

Part {\bf(c)} immediately follows from \cite[Theorem 14.0.1]{meyn2009markov} since condition (iii) of
\cite[Theorem 14.0.1]{meyn2009markov} holds with the choice of $V(x)=f(x)$.

\end{proof}

%%%%%%%%%%%%%
%%%ide
%%%%%%%%%%%%%

%%%%%%%%%%%%%%%%
%%%
%%%%%%%%%%%%

\subsection{The density $p_j(x,y)$}\label{z37}

For an $A\subset \mathbb{R}^+:=(0,\infty ) $ let
$
P_j(x,A):=\mathbb{P}_x\left(Z_j\in A\right)
$
and let $p_j(x,y)$ be the density of $P_j(x,dy)$. When $j=1$ then we suppress the index.
Similarly to \eqref{t137}, using the density of an exponential random variable with parameter $\lambda_j$, we have that
\begin{equation}\label{t174}
  p_j(x,y)=\frac{\lambda_j
  \exp\left(-\lambda_j
  g_{x}^{-1}(y)\right)}
  {{g'_x}(g_{x}^{-1}(y))}.
\end{equation}
\structure{Combining  \eqref{0123}
and the second part of \eqref{t123} we obtain the scaling property}
\begin{equation}\label{t173}
p_j\left(\frac{x}
{\lambda_{j}^{\theta}},
\frac{y}{\lambda_{j}^{\theta}}\right)
=
\lambda_j^\theta \cdot
p_1(x,y).
\end{equation}
Let $\pi_j$ be the stationary distribution  for the chain
$\Phi^{(j)}$. That is $\pi_j$ is defined as the unique finite measure satisfying
\begin{equation}\label{t175}
  \pi_j(A)=\int P_j(x,A)d\pi_j(x)=
  \iint \ind_A(y) \cdot p_j(x,y)d y d\pi_j(x).
\end{equation}
It follows from Theorem \ref{t131} that $\pi_j$ exists and  absolute continuous. Let $\varphi_j$ be its density. Then
\begin{equation}\label{t176}
  \int \varphi_j(x)p_j(x,y)dx=\varphi_j(y).
\end{equation}
It follows from this \structure{identity} and \eqref{t173}
that  for all $j$ and  $u\in \mathbb{R}^+$ and $A\subset \mathbb{R}^+$ that
\begin{equation}\label{t177}
  \varphi_j(u)=\varphi(u\lambda_{j}^{\theta})
   \cdot \lambda_{j}^{\theta} \mbox{ and }
   \pi_j(A)=\pi_0(\lambda_{j}^{\theta} \cdot A).
\end{equation}
From now on we always assume that
$\Phi^{(j)}_{0}$ is chosen according to $\pi_j$.

%%%%%%%%%%%%%%%%%%
%%%%%%%%%%%%%%%%%%%%%%%%%
%%%%%%%%%%%%%%%%%%%%%%%%%%

\section{Some large deviation and variance estimates}

\subsection{Some large deviation results}\label{t234}
Now we estimate the number and length of losses of $Z_j(t)$ on the interval \structure{$[0,1]$}, using some large deviation results.
The numbers of losses on the interval
$[0,1]$ of $Z_j(t)$ for  $j=1,2, \dots $ are given by the independent  $\mathrm{Poisson}(\lambda_j)$ processes. \structure{Let
 $N_j(t)$ be the number of losses of $Z_j(t)$ on $[0,t]$.}
We write
\begin{equation}\label{z45}
  \mathcal{N}_j:=N_j(1)
\end{equation}
Then
$\mathcal{N}_j\sim \mathrm{Poi}(\lambda_j)$.
\structure{Here we define four events that are likely to happen in these Poisson processes. Recall that $T_k^{(j)}$ stands for the time of the $k$th loss while $\tau_k^{(j)}=T_k^{(j)}-T_{k-1}^{(j)}$ is the $k$th inter-event time.}
 \begin{align} \label{e1juli}
E_{1}^{(j)}&:=
\left\{
\frac{\lambda_j}{2}<\mathcal{N}_j \leq 2\lambda_j
\right\}
\\
\label{e2juli}
E^{(j)}_{2}&:=
\left\{
\#\left\{
k: k\le\lfloor\lambda_j/2\rfloor,
\tau_{k}^{(j)}>\frac{1}{2\lambda_j}
\right\} \geq \frac{\lambda_j}{4}
\right\},\\
\label{t176}
  E^{(j)}_{3}&:=
\left\{
\#\left\{k:
k \le 2\lambda_j,
 \frac{1}{100\lambda_j}<
\tau_{k}^{(j)}<\frac{5}{\lambda_j}
\right\} \geq 2\lambda_j\frac{98}{100}
\right\}
\end{align}
%\begin{multline}\label{t176}
%  E^{(j)}_{3}:=
%\left\{
%\#\left\{
%k\in\left\{1, \dots ,2\lambda_j\right\}:
%\right.\right. \\
%\left. \left. \frac{1}{100\lambda_j}<
%\tau_{k}^{(j)}<\frac{5}{\lambda_j}
%\right\} \geq 2\lambda_j\frac{98}{100}.
%\right\}
%\end{aligned}\]

%\end{multline}
Further, we define
\begin{equation}\label{z42}
  E^{(j)}_{4}:=\left\{
 \#\left\{k:
  k\le \mathcal{N}_j, \frac{1}{100\lambda_j}<
  \tau_{k}^{(j)}
  <
  \frac{5}{\lambda_j}
  \right\}>0.97\mathcal{N}_j\right\}.
\end{equation}

Finally we set
\begin{equation}\label{t221}
  E^{(j)}:=E^{(j)}_1\cap E^{(j)}_2\cap E^{(j)}_3\cap E^{(j)}_4.
\end{equation}
\begin{fact}\label{t222}
For almost all realizations $\omega$ \structure{of the Poisson point processes with intensities $\lambda_j, j\ge 1$},
  there exists a $j_0=j_0(\omega)$ such that for all $j>j_0$ the event $E^{(j)}$ holds.
\end{fact}
\begin{proof}
  %It follows from Theorem 2.19 and Exercise 2.221 from  Remco's book \margo{referencia kell} that \margo{add reference}
  It follows from a standard Chernoff bound for Poisson random variables  (see e.g. \cite[Theorem 2.19, Exercise 2.221]{Remcobook}) that
$$
\mathbb{P}\left(\mathcal{N}_j<\frac{\lambda_j}{2}\right)
<\exp(-\lambda_j/8) \mbox{ and }
\mathbb{P}\left(\mathcal{N}_j>2\lambda_j\right)
<\exp(-\lambda_j \cdot 3/8).
$$
By \eqref{o59} both of these series are summable in $j$. Thus, using Borel-Cantelli Lemma we obtain that the event $E^{(j)}_1$ holds for all sufficiently large $j$.
To estimate the probability of ${E_{2}^{(j)}}^c$, the complement of $E_{2}^{(j)}$, from above  note that
$$
{E_{2}^{(j)}}^c=
\left\{
\sum\limits_{k=1}^{\lambda_j/2}
X_k>\lambda_j/4
\right\},
$$
where $X_k=\ind_{\{  \tau_{k}^{(j)}<1/(2\lambda_j)\}}$ is the indicator of the event $\{\tau_{k}^{(j)}<1/(2\lambda_j)\}$, a Bernoulli random variable
%$$
%X_k:=
%\left\{
%  \begin{array}{ll}
%    1, & \hbox{if $\tau_{k}^{(j)}<\frac{1}{2\lambda_j}$;} \\
%    0, & \hbox{if $\tau_{k}^{(j)} \geq \frac{1}{2\lambda_j}$.}
%  \end{array}
%\right.
%$$
with parameter
$p:=\mathbb{E}\left[X_k\right]=\mathbb{P}
\left(\tau^{(1)}<1/2\right)<0.4$.
Then again, from a usual Chernoff bound (see again \cite[Theorem 2.19]{Remcobook}) we obtain that
$$
\mathbb{P}\left({E_{2}^{(j)}}^c\right)
 \leq
\exp\left(
-\lambda_j \cdot \frac{\left(0.5-p\right)^2}{2\left(
p+\left[0.5-p\right]/3
\right)}
\right).
$$
This is also summable by \eqref{o59}.
Similarly, using the same argument for the lower and upper bounds we obtain that
\begin{equation}\label{t276}
  \mathbb{P}\left(
  {E_{3}^{(j)}}^c
  \right)
   \leq
   2\exp\left(-3\lambda_j/100 \right)
\end{equation}
Finally, in the exact same way as above one can easily see that
$ \mathbb{P}\left(
  {E_{4}^{(j)}}^c
  \right)$ is also summable.
Then we use Borel-Cantelli Lemma again to complete the proof of the fact.
\end{proof}

\begin{fact}\label{t223}
  Let
\begin{equation}\label{t224}
  \iota_j:=\frac{\theta+1}{\tilde c} \cdot \frac{ \log\lambda_j}{\lambda_j},
\end{equation}
where the constant $\tilde c>0$ comes from \eqref{t164}.
Further let
\begin{equation*}\label{t225}
  A_{k}^{(j)}:=\left\{\Phi_{k}^{(j)}<\iota_{j}^{\theta}\right\}
\mbox { and }
A^{(j)}:=\bigcap\limits_{k=1}^{2\lambda_j}A_{k}^{(j)}.
\end{equation*}
Then there exists a $j_1$ such that for every $j>j_1$ the event $A^{(j)}$ happens.
\end{fact}
Note that if $E_{1}^{(j)}\cap A^{(j)}$ happens then for all $0 \leq t \leq 1$,  $Z_j(t)<\iota_j^\theta$ holds.
\begin{proof}We compute using the scaling property in \eqref{t177} of the stationary measure $\pi_j$
 \[
 \mathbb{P}\left({A_{k}^{(j)}}^c\right)
 =\mathbb{P}\left(\Phi_{k}^{(j)}>\iota_{j}^{\theta}\right)
 =\pi_j\left(\iota_{j}^{\theta},\infty \right)=\pi_0\left(\lambda_{j}^{\theta}\iota_{j}^{\theta},\infty \right).
 \]
 Decomposing the right hand side into  intervals $(k^\theta, (k+1)^{\theta})$ we estimate
  \[  \mathbb{P}\left({A_{k}^{(j)}}^c\right)
\leq \sum\limits_{k \geq \lambda_j\iota_{j}^{\theta}}\!\!
  \pi_0 \left(k^\theta,(k+1)^\theta\right)\!
    \leq
     \sum\limits_{k \geq \lambda_j\iota_{j}^{\theta}}\!\!
     K_3\exp\left(-c\lambda_j\iota_j\right)
      \leq C  \exp\left(-\tilde c\lambda_j\iota_j\right),\]
 for some constant $C>0$, where we used  Corollary \ref{t165}. Using this estimate we obtain
 \[
    \mathbb{P}\left( {A^{(j)}}^c\right)=
 \mathbb{P}\left(\bigcup_{k=1}^{2\lambda_j} {A_{k}^{(j)}}^c\right) \leq
 2\lambda_j
C  \exp\left(-\tilde c\lambda_j\iota_j\right)
   <
 C  \frac{1}{\lambda_{j}^{\theta}},
 \]
 again for some constant $C>0$.
By \eqref{o59} the rhs is summable. Applying Borel-Cantelli  lemma finishes the proof.
\end{proof}
In what follows we combine Fact \ref{t222} and Fact \ref{t223} to obtain
\begin{corollary}\label{t226}
Almost surely,
there exists  a $j_2$ (which depends on the representation) such that
\begin{equation}\label{t227}
  \forall j>j_2,\ \forall t\in[0,1], \ Z_j(t) \leq {\iota_j^\theta}.
\end{equation}

\end{corollary}

In the next two sections we frequently use
the following fact:
For $r \geq 1$ there exists a constant $c_{12}>0$ such that for any $h \geq 0$
\begin{equation}\label{t183}
  \int\limits_{t=0}^{\infty }
  \left(t+h\right)^r\e{-t}dt
  =
  \left\{
    \begin{array}{ll}
      c_{12}h^r, & \hbox{if $h \geq 1$;} \\
      c_{12}, & \hbox{if $0<h<1$.}
    \end{array}
  \right.
\end{equation}

\subsection{\structure{Expected} increments, assuming no loss}
In this and the next sections we study the increments of $Z_j$
\begin{equation}\label{t178}
  \Delta_{[a,a+h]} Z_j(t):=
  Z_j(a+h)-Z_j(a)
\end{equation}
on a given interval $[a,a+h]$. Since the process is stationary,
\begin{equation}\label{t188}
  \mathbb{E}\left[\Delta_{[a,a+h]} Z_j(t)\right]=0.
\end{equation}
 However, here we assume that there is no event of loss on $\left[a,a+h\right]$. Under this condition in Proposition \ref{t207} below we give an effective upper bound on the expected increment.

%%%%%%%%%%%%%%%%%%%%%%
%%Ide másolva

\begin{proposition}\label{t207}
 For an arbitrary $a>0$, and $j$ such that $\lambda_j h<1$ we have for some constant $C>0$ that
\begin{equation}\label{t208}
  \mathbb{E}\left[\Delta_{[a,a+h]}Z_j |\
Z_j \mbox{ has no loss on } (a,a+h)
\right] \leq C \cdot \frac{\lambda_j h}{\lambda_{j}^{\theta}}.
\end{equation}
\end{proposition}

If $Z_1(t)$ has no loss on the interval $[a,a+h]$, then the distribution of the time of the last loss before $a$ is exponential. Assuming that the value of $Z_1$ \emph{right before the last loss preceding $a$} was  $x$ then the conditional expectation of the increment of $Z_1$ on this interval  is $\int\limits_{t=0}^{a }
\left(
g_x(t+h)-g_x(t)
\right)\e{-t}dt$,  assuming  that there was a loss before $a$ at all. Otherwise the increment is $g_x(a+h)-g_x(a)$, where $x=Z_1(0)$  and this happens with probability $\e{-a}$.  This observation motivates the following two lemmas about the first and second moment of the (conditional) increment.

\begin{lemma}\label{t182}
  Let $\ell =1,2$ and we define
$$
 I_{x,h,\ell}:=
\int\limits_{t=0}^{\infty }
\left(
g_x(t+h)-g_x(t)
\right)^{\ell }\e{-t}dt
$$
Then there exists a constant $c_{17}>0$ such that
\begin{equation}\label{t304}
 I_{x,h,\ell} \leq
c_{17}
\max\left\{x^{\ell (1-1/\theta)+1/\theta},1\right\} \cdot
\max\left\{
h^\ell ,h^{\ell \cdot  \theta}
\right\}
\end{equation}
\end{lemma}

\begin{proof}

Fix $x,h$ and $\ell $.
By the mean value theorem, for every $t$ we pick a $t'\in\left(\frac{t}{x^{1/\theta}}
,
\frac{t+h}{x^{1/\theta}}
\right)$ such that
\begin{equation}\label{t186}
 g_1\left(\frac{t+h}{x^{1/\theta}}\right)
-
g_1\left(\frac{t}{x^{1/\theta}}\right)
=g'_1(t') \cdot \frac{h}{x^{1/\theta}}
\end{equation}
Then we can write
\[\begin{aligned}
  I_{x,h,\ell}&=
 \int\limits_{t=0}^{\infty }
x^\ell \left(
g_1\left(\frac{t+h}{x^{1/\theta}}\right)
-
g_1\left(\frac{t}{x^{1/\theta}}
\right)\right)^\ell \e{-t}dt
 \\
 &=  \underbrace{\int\limits_{t=0}^{\xi x^{1/\theta}-h}
}_{:=I_1}
+
\underbrace{\int\limits_{\xi x^{1/\theta}-h}^{\xi x^{1/\theta}}
}_{:=I_2}
+
\underbrace{\int\limits_{\xi x^{1/\theta}}^{\infty },
}_{:=I_3}
\end{aligned}\]
%hh
%\[ \begin{aligned}
%  I_{x,h,\ell}&=
% \int\limits_{t=0}^{\infty }
%x^\ell \left(
%g_1\left(\frac{t+h}{x^{1/\theta}}\right)
%-
%g_1\left(\frac{t}{x^{1/\theta}}
%\right)\right)^\ell \e{-t}dt
% \\
%&=  \underbrace{\int\limits_{t=0}^{\xi x^{1/\theta}-h}
%}_{:=I_1}
%+
%\underbrace{\int\limits_{\xi x^{1/\theta}-h}^{\xi x^{1/\theta}}
%}_{:=I_2}
%+
%\underbrace{\int\limits_{\xi x^{1/\theta}}^{\infty },
%}_{:=I_3}
%\end{aligned}\]
where we have decomposed the integral into the integral on three disjoint intervals.

\textit{The estimate of $I_1$}:
For $t \leq \xi x^{1/\theta}-h$ we have $t' \leq \frac{t+h}{x^{1/\theta}}<\xi$. Let $R:=\max\limits_{u\in[0,\xi]} g'_1(u)$. \emph{So, it is easy to see that for some constant $C>0$,}
$$
I_1
 \leq
x^\ell  R^\ell h^\ell x^{-\ell /\theta}\xi
x^{1/\theta }
 \leq C \cdot
h^\ell x^{\ell (1-1/\theta)+1/\theta}.
$$
Thus, $I_1$ is at most the right hand side of \eqref{t304}.

\textit{The estimate of $I_2$}: Using \eqref{t186} first we can apply the upper bound on $g_1'$ in \eqref{t167} to obtain:
\[\begin{aligned}
I_2 &\leq \int\limits_{\xi x^{1/\theta}-h}^{\xi x^{1/\theta}}
x^\ell \left(
g'_1\left(t'\right)\frac{h}{x^{1/\theta}}
\right)^\ell \e{-t}dt\\
& \leq x^{\ell (1-1/\theta)} h^{\ell }\psi^\ell \theta^\ell
\int\limits_{\xi x^{1/\theta}-h}^{\xi x^{1/\theta}}
g_{1}^{\ell (\theta-1)/\theta}
(t')\e{-t}dt.
  \end{aligned} \]
%\[ \begin{aligned}
%I_2 &\leq \int\limits_{\xi x^{1/\theta}-h}^{\xi x^{1/\theta}}
%x^\ell \left(
%g'_1\left(t'\right)\frac{h}{x^{1/\theta}}
%\right)^\ell \e{-t}dt\\
% &\leq x^{\ell (1-1/\theta)} h^{\ell }\psi^\ell \theta^\ell
%\int\limits_{\xi x^{1/\theta}-h}^{\xi x^{1/\theta}}
%g_{1}^{\ell (\theta-1)/\theta}
%(t')\e{-t}dt.
%   \end{aligned}\]
  Now we use that $g_1$ is increasing, thus we can use the upper bound on $t'$ from before \eqref{t186} and then apply \eqref{t145} to obtain
   \[ \begin{aligned}
    I_2&\leq x^{\ell (1-1/\theta)} h^{\ell }\psi^\ell \theta^\ell
\int\limits_{\xi x^{1/\theta}-h}^{\xi x^{1/\theta}}
g_{1}^{\ell (\theta-1)/\theta}
\left(\frac{t+h}{x^{1/\theta}}\right)\e{-t}dt\\
 &\leq
 h^{\ell }\psi^ \ell \theta^\ell
\int\limits_{\xi x^{1/\theta}-h}^{\xi x^{1/\theta}}C
(t+h)^{\ell (\theta-1)}\e{-t}dt.
\end{aligned} \]
By \eqref{t183} we have shown
$$
I_2 \leq    \left\{
                              \begin{array}{ll}
                               c_{17} h^{\ell \theta}, & \hbox{if $h \geq 1$;} \\
                               c_{17}  h^\ell, & \hbox{if $h\in (0,1)$.}
                              \end{array}
                            \right.
$$
Note that  the right hand side of \eqref{t185} is an upper bound on the rhs.

\textit{The estimate of $I_3$}:
First we apply the Lagrange mean value theorem
as in \eqref{t186}. It will be important that  $t'>\xi$ since $t>\xi x^{1/\theta}$. Then we use \eqref{t167} and \eqref{t145} in this order to obtain the following upper bound:
\begin{align}\label{t187}
   g_1\left(\frac{t+h}{x^{1/\theta}}\right)
&-
g_1\left(\frac{t}{x^{1/\theta}}\right)
 =
g'(t') \cdot \frac{h}{x^{1/\theta}}
 \leq
\psi\theta\cdot \frac{h}{x^{1/\theta}} \cdot
(g_1(t'))^{1-1/\theta}
 \\
 &\leq
\psi\theta\cdot \frac{h}{x^{1/\theta}} \cdot
\left(\frac{t+h}{x^{1/\theta}}\right)^{\theta-1}=
\psi\theta\cdot x^{-1}h(t+h)^{\theta-1}.
\end{align}
Combining this with \eqref{t123} and \eqref{t183} we obtain that
$$
I_3 \leq \mathrm{const} \cdot \left\{
                              \begin{array}{ll}
                               h^{\ell \theta}, & \hbox{if $h \geq 1$;} \\
                                h^\ell, & \hbox{if $h\in (0,1)$.}
                              \end{array}
                            \right.
$$
which is not greater than the right hand side of  \eqref{t185}.
\end{proof}
A modification of Lemma \ref{t182} is the following lemma.
\begin{lemma}\label{t190}
  Let $\ell =1,2$  and $b>0$ be arbitrary. We define
\begin{equation}\label{t192}
  J_{x,h,\ell,b}:=
\left(
g_x(b+h)-g_x(b)
\right)^{\ell } \cdot \e{-b}.
\end{equation}
Then there exists a constant $c_{18}>0$ such that
\begin{equation}\label{t185}
 J_{x,h,\ell,b} \leq
c_{18}
\max\left\{x^{\ell (1-1/\theta)},1\right\} \cdot
\left\{
  \begin{array}{ll}
    h^\ell, & \hbox{if $h<1$;} \\
   h^{\ell \cdot  \theta} , & \hbox{if $h>1$.}
  \end{array}
\right.
\end{equation}
\end{lemma}
\begin{proof}
 This can be proved exactly as we proved Lemma \ref{t182}. More precisely,  we  separate three cases according to $b<\xi x^{1/\theta}-h$, $\xi x^{1/\theta}-h<b<\xi x^{1/\theta}$ and $b>\xi x^{1/\theta}$. Using the same arguments as above, we obtain that
there exists a constant $c_{30}>0$ such that
\begin{equation}\label{z39}
   \left(
g_x(b+h)-g_x(b)
\right)^{\ell } \cdot \e{-b} \\
   \leq
   c_{30}
\max\left\{x^{\ell (1-1/\theta)},1\right\} \cdot
\left\{
  \begin{array}{ll}
    h^\ell, & \hbox{if $h<1$;} \\
   h^{\ell \cdot  \theta} , & \hbox{if $h>1$.}
  \end{array}
\right.
\end{equation}

\end{proof}

We are ready to prove Proposition \ref{t207}.
\begin{proof}[Proof of Proposition \ref{t207}] Recall that we write $\phi_j$ for the density of the stationary distribution of $Z_j(t)$, and that the time of the last loss before time $a$ has a truncated exponential distribution. Let us write $P_{a}^{(j)}$ for the number of losses (points of the underlying Poisson point process) of $Z_j$ on the interval $[a, a+h]$. We obtain
\begin{align*}
\mathbb{E}\left[\Delta_{[a,a+h]}Z_j|
P_a^{(j)}=0
\right]
&\!\leq\!
\underbrace{ \int\limits_{x=0}^{\infty }
 \lambda_j
\int\limits_{t=0}^{a}
 \left(
 g_x\left(t+h\right)-g_x(t)
 \right)^\ell \e{-\lambda_jt}
 dt \varphi_j(x)dx}_{:=I_1}\\
 &\ +
 \underbrace{ \int\limits_{x=0}^{\infty }
\left(g_x(a+h)-g_x(a)\right)^\ell
 \e{-\lambda_i a}
  \varphi_j(x)dx}_{:=I_2}
\end{align*}
We apply  \eqref{t177} on $\varphi_j$ and then  \eqref{0123} in this order. Then  we make the change or variables in $I_1$ in the most natural way to obtain the upper bound for $I_1$ from Lemma \ref{t182} and Part (b) of Theorem \ref{t131}. We obtain the upper bound for $I_2$  immediately by Lemma \ref{t190} and also Part (b) of Theorem \ref{t131}.
\end{proof}
%%%%%%%%%%%%%%%%%%

\subsection{Variance of the increments}

In this section our aim is to compute the variance of the increments. For the model of TCP RENO this was done in \cite{Barral2004}. Then  we reformulate  a very simple but  useful inequality which was introduced \cite[Lemma VI.1]{Rams2012}. We simply call it Markov inequality and just as in \cite{Rams2012} we use it frequently later on.

 \begin{proposition}\label{t180}
   For every $a>0$ we have
\begin{equation}\label{t179}
\mathrm{Var}\left(  \Delta_{[a,a+h]} Z_j(t)\right)
   \leq
   K_0\min \left\{\frac{1}{\lambda_{j}^{2\theta}},
\frac{h}{\lambda_{j}^{2\theta-1}}
\right\}.
\end{equation}
\end{proposition}

\begin{proof} Recall that $P_{a}^{(j)}$ stands for the number of losses (points of the underlying Poisson point process) of $Z_j$ on the interval $[a, a+h]$.
  Let
$
R_j:=\left\{
P_{a}^{(j)} \ge 1
\right\}.
$
Since the process $Z_j$ is stationary,  we can write
\begin{align*}
  \mathrm{Var}\left(\Delta_{[a,a+h]} Z_j\right)&
=
\mathbb{E}\left[\Delta_{[a,a+h]} Z_{j}^{2}\right]
=
\mathbb{E}\left[
\left(Z_j(a+h)-Z_j(a)\right)^2
\right]\\
& \leq  \underbrace{\mathbb{E}\left[
\left(Z_j(a+h)-Z_j(a)\right)^2 | R_j
\right]\mathbb{P}\left(R_j\right)}_{:=A}\\
&\ +
\underbrace{\mathbb{E}\left[
\left(Z_j(a+h)-Z_j(a)\right)^2 | R^c_j
\right]\mathbb{P}\left(R^c_j\right)}_{:=B}.
\end{align*}
By the definition of $R_j$:
\begin{equation}\label{t189}
  \mathbb{P}\left(R_j\right)=
1-\mathbb{P}\left(R_{j}^{c}\right)=
 1-\e{-\lambda_j h} \leq h\lambda_j.
\end{equation}
Using \eqref{t171} we can bound term $A$ as follows:
\begin{equation}\label{t194}
 A \leq  \frac{1}{\lambda_{j}^{2\theta}} \cdot
\mathbb{E}\left[
Z_1(\lambda_j(a+h))^2+Z_1(\lambda_ja)
^2 |R_j
\right] \cdot \mathbb{P}\left(R_j\right)
\end{equation}
First we observe that
 Theorem \ref{t131} implies that
there is a constant $K_6$ such that
for every $t   \geq  0$  we have
$$
\mathbb{E}\left[Z_1(t)^2\right]<K_6.
$$
Hence, $A \leq 2K_6/(\lambda_{j}^{2\theta}) \min\left\{1,h\lambda_j\right\}$. So we have verified that
\begin{equation}\label{t193}
  A \leq 2K_6\min \left\{\frac{1}{\lambda_{j}^{2\theta}},
\frac{h}{\lambda_{j}^{2\theta-1}}.
\right\}
\end{equation}
To estimate term $B$
we introduce $F_{C_j}(t)$ which is the cumulative distribution function of the current lifetime $C_i(a)$ (the time between $a$ and the last loss before $a$) for a $\mathrm{Poisson}(\lambda_j)$ process, a truncated exponential distribution. That is,
 \begin{equation}\label{t191}
F_{C_j}(t):=\mathbb{P}\left(C_j(a) \leq t\right)
=\left\{
   \begin{array}{ll}
     1-\e{-\lambda_jt}, & \hbox{if $t<a$;} \\
     1, & \hbox{if $t \geq a$ .}
   \end{array}
 \right.
\end{equation}

Set $u(x,t,h):=\left(g_x(t+h)-g_x(t)\right)^2$. Using that $\varphi_j(x)$ is the density of the stationary distribution $\pi_j(x)$.
%and $\varphi_j(x)$ satisfies
%\eqref{t177} yields that
\[ \begin{aligned}
  B&= \mathbb{P} \left(R_{j}^{c}\right)\cdot \mathbb{E}\left[u(x,t,h)| R_{j}^{c}\right]=
\e{-\lambda_jh}
\int\limits_{x=0}^{\infty }
\int\limits_{t=0}^{a}
u(x,t,h)d F_{C_j}(t)d\pi_j(x)\\
&=
\e{-\lambda_jh}
\int\limits_{x=0}^{\infty }
\left[\,
\int\limits_{t=0}^{a}
u(x,t,h)\e{-\lambda_j t}dt+u(x,a,h)\e{-\lambda_j a}
\right]\varphi(x \cdot \lambda_{j}^{\theta})\lambda_{j}^{\theta}dx.
\end{aligned}\]
We switch to the stationary density $\varphi:=\varphi_1$ using \eqref{t177} as well as use the self-similar property of $g_x(t)$ as in \eqref{t123} to transform $u$ we write
\[ \begin{aligned}
B&=\e{-\lambda_jh}
\int\limits_{x=0}^{\infty }
\left[\,
\int\limits_{t=0}^{a}
u\left(
\frac{y}{\lambda_{j}^{\theta}},t,h
\right)\e{-\lambda_jt}dt+\e{-\lambda_ja}
u\left(\frac{y}{\lambda_{j}^{\theta}},t,h\right)
\right]\varphi(y)dy\\
&=
\frac{\e{-\lambda_jh}}{\lambda_{j}^{2\theta}}\int\limits_{y=0}^{\infty }\!
\left[\,
\int\limits_{t=0}^{a}
u(y,\lambda_jt,\lambda_jh)\e{-\lambda_jt}dt
+
\e{-\lambda_ja}
u\left(y,\lambda_ja,\lambda_jh\right)
\right]\varphi(y)dy\\
&\leq\!
 \frac{\e{-\lambda_jh}}{\lambda_{j}^{2\theta}\lambda_{\min}}\!
\int\limits_{y=0}^{\infty }
\left[
\int\limits_{v=0}^{\lambda_ja}
u\left(y,v,\lambda_jh\right)\e{-v}dv
\!+\!
u\left(y,\lambda_ja,\lambda_jh\right)\e{-\lambda_ja}
\right]\!\varphi(y)dy
\end{aligned}\]
Using the notation of Lemmas \ref{t182} and \ref{t190} and then using the assertions of Lemmas \ref{t182} and \ref{t190} we continue as follows
\[\begin{aligned}
B&=
\frac{\e{-\lambda_jh}}{\lambda_{j}^{2\theta}}
\frac{1}{\lambda_{\min}}\cdot\int\limits_{y=0}^{\infty }
\left(I_{y,\lambda_jh,2}+J_{y,\lambda_jh,2,\lambda_ja}\right)
d\pi(y)\\
& \leq
\frac{\e{-\lambda_jh}}{\lambda_{j}^{2\theta}}
\frac{1}{\lambda_{\min}}
\cdot\int\limits_{y=0}^{\infty }
\left(
c_{19}\max\left\{y^{2(1-1/\theta)+1/\theta},1\right\}\max\left\{(\lambda_jh)^{\ell },
(\lambda_jh)^{\ell\theta }\right\}
\right)
d\pi(y)
\\
 &\leq
\frac{1}{\lambda_{j}^{2\theta}}
\frac{1}{\lambda_{\min}}\cdot
 \lambda_jh \cdot\underbrace{\e{-\lambda_jh} \cdot
\left((\lambda_jh)+
(\lambda_jh)^{2\theta-1}\right)}_{:=q_j(\lambda_jh)}
\\
& \leq K_{23} \cdot
\min\left\{ \frac{1}{\lambda_{j}^{2\theta}},
\frac{h}{\lambda_{j}^{2\theta-1}}
\right\},
\end{aligned}\]
where $K_{23}$ is  $\max_{x}q_j(x)$ if $\lambda_jh \leq 1$ and $K_{23}$ is  $\max_{x \geq 1}xq_j(x)$ if $\lambda_jh>1$.
We choose $K_0$ as the maximum of $2K_6$ and
$K_{23}$ to complete the proof.
\end{proof}

An immediate corollary of Markov's inequality is the following assertion that we will call Markov inequality in the note.

\begin{lemma}\label{o72}[Markov's inequality]
Given  $n\in\mathbb{N}$ and events $A_1, \dots ,A_n$ such that  for all $i=1, \dots ,n$ we have
$$\mathbb{P}\left(A_i\right) \leq \hat{p}$$
 for some $\hat{p}\in[0,1]$. Then for any $N>1$
\begin{equation}\label{z40}
  \mathbb{P}\left(\#\left\{k \leq n:A_k \text{ happens }\right\}>N \cdot n\hat{p}\right)<\frac{1}{N}.
\end{equation}
\end{lemma}

\section{Technical Lemmas}

Fix an $\ell $ and a $0 \leq k<2^\ell $. Here as well as throughout the paper we write $h:=2^{-\ell }$.
For an $r>0$ to be specified later, we divide the increments $\Delta_{\ell }^{k} Z_j$ of $Z_j$ on the $2^\ell $-mesh intervals
 (defined in \eqref{z3}) into two groups depending on their intensity:
\begin{equation}\label{o122}
\underline{T}^{k,\ell }_{r}:=\sum\limits_{\lambda^{-\theta} _j<r}\Delta_{\ell }^{k} Z_j,\quad
\overline{T}_{r}^{k,\ell }:=\sum\limits_{\lambda^{-\theta} _j\geq r}\Delta_{\ell  }^{k} Z_j.
\end{equation}
To simplify the notation we suppress the super indexes when we consider $k$ and $\ell $ fixed.
We remark that the first sum is the combined effect of increments that are generally small and the second one is the combined effect of  increments that are expected to be large. Namely, the typical magnitude of $Z_j$ is $\lambda_{j}^{-\theta}$ and its typical increments are also $\lambda_{j}^{-\theta}$ if there is a loss on the interval under consideration.
More precisely, it follows from \eqref{t143} and \eqref{t123} that
\begin{equation}\label{z43}
  g_x(u) \geq c_1 u^\theta,\qquad \forall u>0,\, t>0.
\end{equation}
The following events will appear frequently the sequel:
\begin{equation}\label{t228}\qquad
  A_{\text{sum}}^\alpha(k,\ell ):=\left\{|\Delta_{\ell }^{k} Z| \geq h^\alpha \right\}
\end{equation}
and
\begin{equation}\label{z44}
  \structure{\underline{A}^\alpha(k,\ell ):=\left\{|\underline{T}_{h^{-\alpha}}^{k,\ell }| \leq \frac{h^\alpha}{2}\right\}}\text{ and }
\structure{\overline{A}^\alpha(k,\ell ):=\left\{|\overline{T}_{h^{-\alpha}}^{k,\ell }| \leq \frac{h^\alpha}{2}\right\}}.
\end{equation}

\begin{fact}\label{t202}Let $\varepsilon_0>0$.
 Then
  there is a  positive constant $c(\varepsilon_0)$  such that for every $k<2^\ell $
  \begin{equation}\label{t203}
    \mathrm{Var}\left(\underline{T}^{k,\ell }_{h^{-\alpha}}\right) \leq
     c(\varepsilon_0) \cdot
      h^{1+\alpha(2+(1-1/\theta)-\beta-\varepsilon_0)}.
  \end{equation}
\end{fact}
\begin{proof} Recall that $L_k=L^k$ for some constant $L>1$. Since the processes $Z_j$ are independent of each other for different values of $j$, we can write
 \[ \mathrm{Var}\left(\underline{T}_{h^{-\alpha}}^{k,\ell }\right) \leq \sum\limits_{\left\{k:L_k>h^{-\alpha}\right\}}
     \sum\limits_{\lambda_{j}^{\theta}\in[L_{k-1},L_k)} \mathrm{Var}\left(
     \Delta_{\ell }^{k}Z_j\right)\]
  Then, recall that $N_k$ stands for the number of indices $j$ with $\lambda_j^\theta$ in the interval $[L_{k-1}, L_k)$, as in \eqref{def:mk-nk}. We estimate $N_k$ using \eqref{o81} to obtain
  \[ \begin{aligned}
    \mathrm{Var}\left(\underline{T}_{h^{-\alpha}}^{k,\ell }\right)&\leq   \sum\limits_{\left\{k:L_k>h^{-\alpha}\right\}}N_k \cdot
     \max\limits_{\lambda_{j}^{\theta} \geq L_{k-1}}\mathrm{Var}(\Delta_{\ell }^{k}Z_j)\\
    & \leq
     \sum\limits_{\left\{k:L_k>h^{-\alpha}\right\}}
     K_1(\varepsilon_0)K_0L_{k}^{\beta-1+\varepsilon_0}h\frac{L}{L_{k-1}^{2-1/\theta}}.
   \end{aligned}  \]
 After rearranging terms we obtain that the exponent of $L_k$ is $\beta-2+\varepsilon_0-(1-1/\theta)<0$, thus the sum can be estimated by the first term. We obtain
    \[  \begin{aligned} \mathrm{Var}\left(\underline{T}_{h^{-\alpha}}^{k,\ell }\right)&\leq
       \sum\limits_{\left\{k:L_k>h^{-\alpha}\right\}}
     K_1(\varepsilon_0)K_0L^{3-1/\theta}hL_k^{\beta-2+\varepsilon_0-(1-1/\theta)}\\
      &\leq c(\varepsilon_0) \cdot
      h^{1+\alpha(2+(1-1/\theta)-\beta-\varepsilon_0)}.
 \end{aligned} \]
 This finishes the proof.
\end{proof}

\subsection{Estimates  on the region $R_1$} In this section we make some preparation to determine the multifractal spectrum on the region $R_1$ as defined in \eqref{z20}. That is, we assume that
\begin{equation}\label{t204}
  \alpha <\frac{1}{\beta-(1-1/\theta)}.
\end{equation}
In this section we always assume that $\varepsilon_0>0$ satisfies that
\begin{equation}\label{t215}
  \alpha(\beta-(1-1/\theta))+\varepsilon_0(\alpha+1)<1.
\end{equation}
To understand the aim of the following assertions recall that
on region $R_1$ our aim  is to verify that
\begin{equation}\label{z41}
  \#\left\{k:|\Delta_{\ell }^{k}Z|\sim h^\alpha\right\}
  \approx
  h^{- \alpha\left(\beta-(1-1/\theta)\right)}.
\end{equation}
Recall the event $\underline{A}^{\alpha}{(k,l)}$ from \eqref{z44}. First note that
 it follows from
Chebyshev's  inequality, Fact \ref{t202} and \eqref{t188} that
for every $0 \leq k<2^\ell $ we have
\begin{equation}\label{o123}
 \mathbb{P}\left(\underline{A}^{\alpha}{(k,\ell )}\right)
  \geq
  1-4c(\varepsilon_0) \cdot h^{1-\alpha(\beta-(1-1/\theta)+\varepsilon_0)}.
\end{equation}

%%%%%%%%%%%%%%%%%%

To shorten the presentation we introduce the event
$$
F:=\left\{
\exists j:\lambda_j^\theta \leq h^{-\alpha},
Z_j \mbox{ has loss in the interval } I_{\ell}^{k}
\right\},
$$
in other words $F$ is the event that one of the processes in $\overline{T}_{h^{-\alpha}}^{\ell,k}$ has a jump (an event of loss) in the interval $I_{\ell}^{k}$.
Recall $A_{\text{sum}}^{\alpha}(k,\ell))$ from \eqref{t228}. Using \eqref{o123} we can estimate $\mathbb{P}(A_{\text{sum}}^{\alpha}(k,\ell))$ by conditioning on $\underline{A}^\alpha(k,\ell)$.
Namely,

\begin{fact}\label{t209} Assume that $\varepsilon_0>0$ satisfies \eqref{t215}. Let $c(\varepsilon_0)$ be the constant introduced in Fact \ref{t202}. Then for every $k,\ell $:
\begin{equation}\label{t210}
  \mathbb{P}\left(A_{\mathrm{sum}}^\alpha(k,\ell )\right) \leq
7c(\varepsilon_0)h^{1-\alpha(\beta-(1-1/\theta))-
\varepsilon_0(\alpha+1)}
\end{equation}
\end{fact}

\begin{proof}
We fix an $\alpha$ and  $0 \leq k<2^\ell $ and we suppress them below, that is, we write $A_{\text{sum}}:=A_{\text{sum}}^{\alpha}(k,\ell)$ and $\underline{A}:=\underline{A}^\alpha(k,\ell).$
\begin{equation}\begin{aligned}\label{o124}
\mathbb{P}\left(A_{\text{sum}}\right) &\leq \mathbb{P}\left(
\underline{A}^{c}\right)+\mathbb{P}\left(A_{\text{sum}}\cap \underline{A}\right)   \\
   &\leq 4c(\varepsilon_0) \cdot h^{1-\alpha(\beta -(1-1/\theta)-\varepsilon_0)} +\mathbb{P}(A_{\text{sum}}\cap \underline{A}).
\end{aligned}\end{equation}
It is left to estimate $\mathbb{P}(A_{\text{sum}}\cap \underline{A})$.
Clearly, $|\Delta Z|=|\underline{T}_{h^{-\alpha}}+\overline{T}_{h^{-\alpha}}|$.
Hence
\begin{equation}
\begin{aligned}\label{t230}
 \mathbb{P}\left(A_{\text{sum}}\cap \underline{A}\right)
 &\leq
 \mathbb{P}\left(
 \overline{A}^c
 \right)
 =
\mathbb{P}\left(|\overline{T}_{h^{-\alpha}}|
>h^{\alpha}/2\right)\\
 &\leq
 \mathbb{P}\left( F\right)+
\structure{\mathbb{P}
 \left(\overline{T}_{h^{-\alpha}}>\frac{h^\alpha}{2}|
F^c\right)}
\end{aligned}
\end{equation}
First we estimate $\mathbb P(F)$ in \eqref{t230}. We decompose the indices $j$ according to which exponential interval $[L_{k-1}, L_k)$ their parameter $\lambda_j^\theta$ falls in, then count the indices $j$ in each interval using the estimate on $N_k$ as in \eqref{z20}. We also use  that $1-\e{-x}\le x$ to obtain that
\begin{equation}
\begin{aligned}
\mathbb P(F)&\le \sum\limits_{\lambda_{j}^{\theta}
<h^{-\alpha}}
  (1-\e{-h\lambda_j})
   \le \sum\limits_{k:\ L_{k-1}<h^{-\alpha}}
\sum\limits_{\lambda_{j}^{\theta}
\in(L_{k-1},L_k)}h\lambda_j\\
&\le \sum\limits_{k:\ L_{k-1}<h^{-\alpha}}
 N_k hL_k^{1/\theta}
 \le \sum\limits_{k:\ L_{k-1}<h^{-\alpha}}
K_1(\varepsilon_0)L_{k}^{\beta-1+\varepsilon_0}
hL_k^{1/\theta}.
\end{aligned}
\end{equation}
Since the exponent of $L_k$ is $\beta-1+\varepsilon_0>0$, the terms in the sum grow exponentially, and hence the sum can be estimated from above by some constant times the last term. Thus we arrive at
\begin{equation}\label{redterm}
\mathbb P(F) \le C(\varepsilon_0) h^{1-\alpha(\beta-(1-1/\theta))-
\alpha\varepsilon_0}.
\end{equation}
Now we turn to estimate $\mathbb{P}
 \left(\overline{T}_{h^{-\alpha}}>\frac{h^\alpha}{2}|
F^c\right)$ in \eqref{t230}. We write $F_j^c:=\ Z_j \mbox{\,has no loss on }I_{\ell}^{k}$ below.
Applying Markov's inequality, we have the upper bound
\begin{equation}
\begin{aligned}
\mathbb{P}
 \left(\overline{T}_{h^{-\alpha}}>\frac{h^\alpha}{2}|
F^c\right)&\le \frac{\mathbb{E}\left[
\overline{T}_{h^{-\alpha}}
|F^c
\right]}{h^\alpha/2}
\le \frac{\sum\limits_
{\lambda_{j}^{\theta}<h^{-\alpha}}
\mathbb{E}\left[
\Delta_{\ell}^{k}Z_j|F_j^c
\right]}{h^\alpha/2}.
\end{aligned}
\end{equation}
Here, we would like to use Proposition \ref{t207}. For this we need to use that $\lambda_j h<\lambda_j h^{\alpha/\theta}<1$ to be able to apply Proposition \ref{t207}. This is the place where we actually use that $\alpha<1/\theta$ (which is always the case when $\alpha(\beta-(1-1/\theta))<1$ that is we are in $R_1$). Decomposing the indices $j$s again,
\begin{equation}
\begin{aligned}
\mathbb{P}
 \left(\overline{T}_{h^{-\alpha}}>\frac{h^\alpha}{2}|
F^c\right)&
\le \frac{ \sum\limits_{k:\ L_{k-1}<h^{-\alpha}}
 \sum\limits_{\lambda_{j}^
{\theta}
\in(L_{k-1},L_k\wedge h^{-\alpha})}
 K'_3\lambda_{j}^{1-\theta}h
 }
 {h^{\alpha}/2} \\
 & \le \frac{\sum\limits_{k:\ L_{k-1}<h^{-\alpha}}
K_1(\varepsilon_0)L_{k}^{\beta-1+\varepsilon_0}
K'_3L_{k}^{1/\theta-1}h/L
}
{h^\alpha/2}
\end{aligned}
\end{equation}
Note that here the exponent of $L_k$ is $\beta-1+\varepsilon_0 + (1/\theta-1)>0$. Thus again, the summands form a geometric series with mean greater than $ 1$, so the sum is constant times the last element.
\begin{equation}\label{blueterm}
\mathbb P \left(\overline{T}_{h^{-\alpha}}>\frac{h^\alpha}{2}|
F^c\right) \le C(\varepsilon_0) h^{1-\alpha(\beta-(1-1/\theta))-
\alpha\varepsilon_0}.
\end{equation}
Combining the estimates \eqref{redterm} and \eqref{blueterm}, combined with \eqref{t230} and then with \eqref{o124} finishes the proof of Fact \ref{t209}.
\end{proof}
An immediate consequence is the following
\begin{corollary}\label{t265}
 Almost surely,  for all $\ell $ large enough we have
 \begin{equation}\label{t271}
  \#\left\{k \leq 2^\ell:
\underbrace{ |
  \Delta_{\ell }^{k}Z |
 \geq h^\alpha}_{A_{\mathrm{sum}}^{\alpha}(k,\ell )}
  \right\}  < 7c(\varepsilon_0)h^{-\alpha(\beta-(1-1/\theta))
  -\varepsilon_0(1+\alpha)},
\end{equation}
 and
 \begin{equation}\label{t270}
   \#\left\{k \leq 2^\ell:
 \underbrace{ |\underline{T}_{h^{-\alpha}}^{k,\ell }|>h^\alpha/2}_{\left(\underline{A}^{\alpha}(k,\ell )\right)^c}
  \right\}
  < 4c(\varepsilon_0)h^{-\alpha(\beta-(1-1/\theta))
  -\varepsilon_0(1+\alpha)},
 \end{equation}
finally,
\begin{equation}\label{t272}
  \#\left\{k \leq 2^\ell:
  \underbrace{  |\overline{T}_{h^{-\alpha}}^{k,\ell }|>h^\alpha/2}_{\left(\overline{A}^{\alpha}(k,\ell )\right)^c}
  \right\}  < 3c(\varepsilon_0)h^{-\alpha(\beta-(1-1/\theta))
  -\varepsilon_0(1+\alpha)}.
\end{equation}
\end{corollary}

\begin{proof}
First we apply
 Lemma \ref{o72}.
with $N=h^{-\varepsilon_0}$, \newline $\hat{p}=
rc(\varepsilon)
h^{1-\alpha(\beta-(1-1/\theta))-\alpha\varepsilon_0}$, where $r=7$, $r=4$ and $r=3$ in the first, second and third case respectively. Since $\left\{h^{-\varepsilon_0}=2^{\varepsilon_0 \cdot \ell }\right\}$ is summable we obtain the assertions from Borel-Cantelli lemma.
\end{proof}

\section{$f_g$ on region $R_1$}\label{z105}
In this section we are ready to prove the upper and lower bound on the large deviation multifractal spectrum $f_g$ on the region $R_1$.
\subsection{Upper bound on $f_g$ on the region $R_1$}

\begin{fact}\label{t211}
  The large deviation spectrum $f_g(\alpha)$ (defined in \eqref{t212})
satisfies
\begin{equation}\label{t213}
  f_g(\alpha) \leq \alpha(\beta-(1-1/\theta))
\end{equation}
\end{fact}

\begin{proof}
  Fix a small $\varepsilon_0>0$ satisfying \eqref{t215}.
By \eqref{o121} we have
\begin{equation}\label{t214}
  N_{\ell }^{\varepsilon_0}(\alpha)
 \leq \overline{N}_{\ell }^{\varepsilon_0}(\alpha):=
\#\left\{
0 \leq k <2^\ell :
\left|\Delta_{ \ell }^{k}Z\right|>
h^{-(\alpha+\varepsilon_0)}
\right\}.
\end{equation}
Now we replace $\alpha$ with $\alpha+\varepsilon_0$ in \eqref{t271}.
 This yields that for almost all realizations, there is an $\ell ^*$ such that for $\ell >\ell ^*$ we have
\begin{equation}\label{t218}
 \overline{N}_{\ell }^{\varepsilon_0}(\alpha)
 <
2^{\ell \left[
(\alpha+\varepsilon_0)(\beta-(1-1/\theta))+\varepsilon_0(1+\alpha)
\right]
}
\end{equation}
Taking logarithm of both sides we have for all
$\ell >\ell ^*$:
$$
\frac{\log \overline{N}_{\ell }^{\varepsilon_0}(\alpha)}{\log 2^\ell }
<
(\alpha+\varepsilon_0)(\beta-(1-1/\theta))
+\varepsilon_0(1+\alpha).
$$
This implies that \eqref{t213} holds.
\end{proof}

%%%%%%%%%%%%%%%%%%%%%
%%új
%%%%%%%%%%%%%%%%%%%

\subsection{The lower bound on $f_g$ on $R_1$}

In this case we need to assume regularity (introduced in Definition \ref{t220}) of the sequence $(\lambda_j)_{j\ge 1}$. Let $\varepsilon_0$ be fixed satisfying \eqref{t215} and
let $A$ be an upper bound on $a_{i+1}-a_i$ from \eqref{ajuli}. That is
$$
0<a_{i+1}-a_i<A\quad  \mbox{ for all } i \geq 2.
$$
Our aim is to prove
\begin{equation}\label{o126}
f_g(\alpha) \geq \alpha(\beta-(1-1/\theta)) \mbox{ if } (\alpha,\beta)\in R_1.
\end{equation}
Further, we always  assume that  $\ell $ is so large that for $h=2^{-\ell }$ we have
\begin{equation}\label{t235}
  L^{A+1}h^{-(\alpha-\varepsilon_0)}\ll h^{-\alpha}.
\end{equation}

\begin{assumption}\label{t239}
  Let $\ell _0$ be so large that besides \eqref{t235} the following statements hold for all $\ell>\ell_0$: whenever
$\lambda_j^\theta >2^{\ell(\alpha-\varepsilon_0)}=h^{-(\alpha-\varepsilon_0)}$ then we have:
\begin{description}
  \item[(a)] $ j>\max{\left\{j_0,j_1,j_2\right\}}$, where  $j_0, j_1, j_2$  were defined in Section \ref{t234},
  \item[(b)] $\iota_j<\left(\frac{1-\theta}{10 \cdot c_{33}}\right)^{\alpha/(\theta-\alpha)}$, where $\iota_j$ was defined in \eqref{t224} and $c_{33}:=\psi\theta\left(\frac{\left(w+1\right)^\theta L^{A+1}}{c_1}\right)^{1/(2\alpha)}$ and $w$ is defined below in \eqref{z86},
  \item[(c)]
  $h^{\varepsilon_0}
  <
  \frac{4}{5}(1-\eta)
  \frac{2c_1}{(w+1)^\theta L^{A+1}}
  $.
\end{description}
\end{assumption}
Fix an
$\ell >\ell _0$.  Recall that the sequence $(\lambda_j)_{j\ge 1}$ is regular, as in Definition \ref{t220}.
By the definition of $A$ in \eqref{ajuli} there exists an  $r$
such that $L_{a_r}\in \left(L \cdot h^{-(\alpha-\varepsilon_0)},L^{A+1}
h^{-(\alpha-\varepsilon_0)}\right)$.
That is for
\begin{equation}\label{t255}
  J_\ell :=\left\{j:
\lambda_{j}^{\theta}\in
\left(h^{-(\alpha-\varepsilon_0)},L^{A+1}
h^{-(\alpha-\varepsilon_0)}\right)
\right\}
\end{equation}
we have
\begin{equation}\label{t241}
  \#J_\ell > h^{(-\alpha+\varepsilon_0)(\beta-1-\varepsilon_0)}.
\end{equation}
In \eqref{o123} we verified that for $\underline A=\underline A^\alpha(k,\ell)$
$$
\mathbb{P}\left(\underline{A}^{c}\right)=
\mathbb{P}\left(
|\underline{T}_{h^{-\alpha}}|>\frac{h^\alpha}{2}
\right) \leq 4C(\varepsilon_0)h^{1-\alpha(\beta-(1-1/\theta))
-\varepsilon_0(1+\alpha)}.
$$
Recall the definition $\underline A=\underline{A}^\alpha(k,\ell)$ from \eqref{z44}.
It is immediate from the bounds following \eqref{t230} that
for any $j\in J_\ell$ and
\begin{equation}\label{t261}
  \overline{B}_{ j}:=\overline{B}_{j}(k,\ell ):=\left\{\left|
\overline{T}^{k,\ell }_{h^{-\alpha}}-\Delta_{\ell }^{k} Z_j
\right|<\frac{h^\alpha}{2}\right\}.
\end{equation}
we have
\begin{equation}\label{t257}
  \mathbb{P}\left(\overline{B}_{j}^{c}\right) \leq
3C(\varepsilon)h^{1-\alpha(\beta-(1-1/\theta))-
\varepsilon_0(\alpha+1)}
\end{equation}
Therefore exactly as in Corollary \ref{t265} we get
\begin{fact}\label{t240}
Almost surely,
there exists an $\ell _1 \geq \ell _0$ (depending on the realization) such that $\forall \ell >\ell _1$ and
$j\in J_\ell $
 we have
  \begin{equation}\label{t233}
  \#\left\{k<2^\ell : \underline{A}(k,\ell )\cap \overline{B}_{j}(k,\ell ) \right\}
  >h^{-1}- 7c(\varepsilon_0)h^{-\alpha(\beta-(1-1/\theta))-\varepsilon_0(\alpha+1)}
.
\end{equation}
\end{fact}
Recall that $T_u^{(j)}$ is the time of the $u$th jump in the Poisson process with intensity $\lambda_j$.
Fix an $\ell >\ell _1$ and $j \in J_\ell$. Let $k:=k_j(u)$ be the index of the $2^\ell $-mesh interval that contains $T_{u}^{(j)}$:
$$
k_j(u):=k \mbox{ if }  T_{u}^{(j)}\in I_{\ell }^{k}.%=\left[\frac{k}{2^\ell },\frac{k+1}{2^{\ell }}\right],
$$
%where we defined earlier
%$T_{u}^{(j)}$ as  the $u\in\mathcal{N}_j$-th loss of $Z_j$
Let
\begin{equation}\label{z89}
  Q_j:=\left\{k :\exists u\in\mathbb{N}
 \mbox{ with }k=k_j(u) \text{ s.t. }T_{u}^{(j)}<1
\right\}
\end{equation}
be the set of the $2^\ell$-mesh intervals where the Poisson process with intensity $\lambda_j$ has jumps.
Recall that $\mathcal{N}_j=\#Q_j$, where $\mathcal N_j$ was defined in \eqref{z45}.
Further we collect the indices $k$ such that we can find a process $Z_j$, $j\in J_\ell$ in such a way that in $I_{\ell }^{k}$ there is `possibly large' loss of $Z_j$, reflected in the fact that the inter-event time $\tau_u^{(j)}=T_u^{(j)}-T_{u-1}^{(j)}$ is sufficiently large:
\begin{equation}\label{t237}
 \mathcal{I}_j^w:=
 \left\{k\corange{\in Q_j} :\exists u\in\mathbb{N}
 \mbox{ with }k=k_j(u) \text{ s.t. }
\tau_{u}^{(j)}>\frac{1}{w\lambda_j}
\right\},
\end{equation}
where $w$ is so big that
\begin{equation}\label{z86}
  \e{-1/w}>0.99.
\end{equation}
Then it follows from the Large Deviation Theorem and Borel Cantelli lemma that for all $\ell $ large enough we have
\begin{equation}\label{z87}
  \#\left\{r \leq \mathcal{N}_j:\tau_{r}^{(j)}>\frac{1}{w\lambda_j}\right\}
  =\#\mathcal{I}_j^w
  >\mathcal{N}_j\e{-1/(2w)}.
\end{equation}
%\margo{0325}
Using \eqref{t233}, also for all $\ell $ large enough we have
\begin{equation}\label{z88}
  \#\left\{k:\underline{A}(k,\ell )\cap \overline{B}_j(k,\ell ) \mbox{ holds}\,\right\}
 \geq h^{-1} \cdot \e{-1/(2w)}.
\end{equation}
Observe that the events that $k\in \mathcal{I}_j$ and the event that $ \underline{A}^\alpha(k,\ell )\cap \overline{B}_{ j}(k,\ell )$ holds are independent (see \eqref{z44} and \eqref{t261} for the definitions).
This is so, because $\overline{B}_{ j}(k,\ell )$ excludes the contribution of $Z_j$, while in $\underline{A}^\alpha(k,\ell )$ only processes from $\underline{T}_{h^{-\alpha}}^{k,\ell }$ can contribute. On the other hand, we have assumed that $j \in J_\ell$ and thus $j$ does not belong to $\underline{T}_{h^{-\alpha}}^{k,\ell }$. % nor in $\overline{T}_{h^{-\alpha}}^{k,\ell }$.

 Combining  \eqref{z87} and \eqref{z88} for
\begin{equation}\label{t260}
  \widetilde{\mathcal{I}}_j^w:=\left\{k\in \mathcal{I}_j^w:
  \underline{A}(k,\ell )\cap \overline{B}_{j}(k,\ell )\mbox{ holds\,}
  \right\}
\end{equation}
yields that for an $\ell $ large enough,
\begin{equation}\label{t259}
  \# \widetilde{\mathcal{I}}_j^w
  \geq \mathcal{N}_j \cdot \e{-1/w}.
\end{equation}
Note that we expect that for any $k\in \widetilde{\mathcal{I}}_j^w
$ will have a sufficiently large increment of $\Delta Z$, since (1) the process $Z_j$ did not jump for a while already, thus it had enough time to increase and thus sustain a sufficiently large loss (2) the increment coming from processes with small intensities is rather small, i.e., $\underline A$ holds (see \eqref{t270}) (3) the processes with relatively large intensities, excluding $Z_j$, have a small increments, i.e., $\overline B_j$ also holds (see \eqref{t261}). The next fact makes the heuristics of (1) precise: any $k\in \mathcal I_j^w$ will actually have a large loss of $Z_j$:
\begin{fact}\label{t242}
Fix an arbitrary $k\in \mathcal{I}_j^w$.
Then
\begin{equation}\label{t250}
  \Delta_{\ell }^{k}Z_j<-2 \cdot h^\alpha.
\end{equation}

\end{fact}
\begin{proof}[Proof of Fact \ref{t242}]
We use the notation of Figure \ref{z47}.
  \begin{figure}[H]
  \includegraphics[width=11cm]{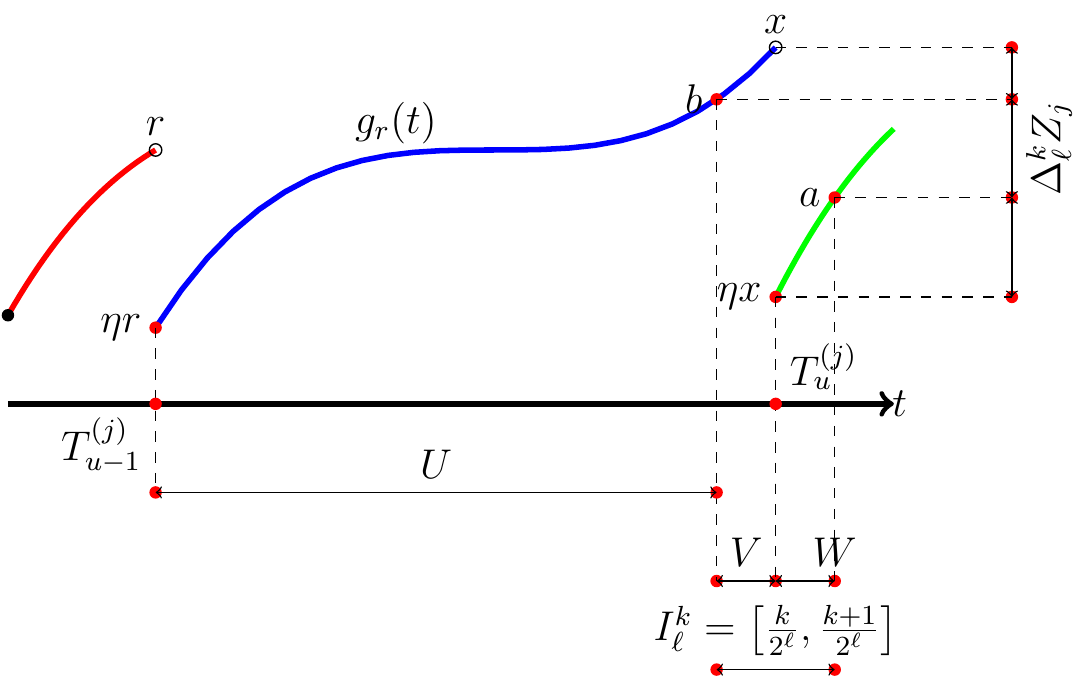}
\caption{Important notation for a process $Z_j$ for the proof of the spectrum on the region $R_1$.}\label{z47}
\end{figure}
Since we are on region  $R_1$, $\alpha<\theta$, therefore
\begin{equation}\label{t245}
h\ll h^{(\alpha-\varepsilon_0)/\theta}L^{-(A+1)/\theta}
 < \lambda_{j}^{-1}
 <
 h^{(\alpha-\varepsilon_0)/\theta}.
\end{equation}
So, the length of interval $U=k/2^\ell - T_{u-1}^{(j)}$ satisfies $|U|>\frac{1}{4\lambda_j}$ (see Figure \ref{z47}). Using this, \eqref{z48} and \eqref{t245}, we obtain that $x$, the position of the process $Z_j$ right before the loss at time $T_u^{(j)}$ satisfies
\begin{equation}\label{t246}
 x \geq\frac{c_1}{4^\theta L^{A+1}}h^{\alpha-\varepsilon_0}.
\end{equation}
We denote $b:=Z_j(k/2^\ell)$ and $a:=Z_j(k+1/2^\ell)$. Then $\Delta_{\ell }^{k}Z_j$ can be estimated as follows
\begin{equation}\label{t244}
  \Delta_{\ell }^{k}Z_j=a-b=-x +a+(x-b)
\end{equation}
We have already estimated $-x$ from above in
\eqref{t246}.

\emph{To estimate $a$}: Note that by \eqref{t123} $a \leq g_x(h)=x g_1\left(h/x^{1/\theta}\right)$.
 Note, by \eqref{t246},
 \[h/x^{1/\theta} \leq c_{1}^{-1/\theta}4L^{(A+1)/\theta}h^{1-\alpha/\theta
 +\varepsilon_0/\theta
 }.\] This expression is sufficiently small if $\ell $ is large, that is $h=2^{-\ell }$ is small. Using that $g_1(0)=\eta$ for $\ell $ large enough we get
 $$
 a \leq x\left(\eta+\frac{1-\eta}{10}\right).
 $$

\emph{To estimate $x-b$:}  Observe that by \eqref{t238}
we have
\begin{equation}\label{t247}
  \max\limits_{x\in V}\frac{d}{dt}g_r(t) \leq \psi\theta x^{1-1/\theta},
\end{equation}
where the interval $V$ is defined on Figure \ref{z47}.
Combining this  with \eqref{t246} we obtain that
\begin{equation}\label{t248}
 x-b \leq h \cdot   \max\limits_{x\in V}\frac{d}{dt}g_r(t) \leq
 \mathrm{const} \cdot x^{1-1/\theta+1/(\alpha-\varepsilon_0)}
\end{equation}
Using Assmption \ref{t239} (b)  we get
\begin{equation}\label{t249}
  x-b \leq x \cdot \frac{1-\eta}{10}.
\end{equation}
 Combining the three estimates on $-x$, $a$ and $x-b$ and using  Assumption \ref{t239} (c) we obtain that \eqref{t250} holds.

\end{proof}

Combining \eqref{t228}, \eqref{t261} and \eqref{t250} yields
\begin{equation}\label{t262}
\forall j\in J_\ell ,\   \forall k\in \widetilde{\mathcal{I}}^w_j,\quad
  \Delta_{\ell }^{k}Z<-h^\alpha.
\end{equation}
Now, $\widetilde I_j$ is the set of $2^{-\ell}$-mesh intervals for the process $Z_j, j \in J_\ell$ with a certain good property. Now we take the union of these intervals for different $j$-s to obtain all the $2^{-\ell}$-mesh intervals that contain this good property for some $j\in J_\ell$. Let $K_\ell$ be their number, that is,
\begin{equation}\label{k-ell} K_\ell :=\#\bigcup\limits_{j\in J_\ell }\widetilde{\mathcal{I}}^w_j.\end{equation}
\begin{fact}\label{t253}
  If $\ell $ is large enough then for $h=2 ^{-\ell }$ we have

\begin{equation*}%\label{t254}
  K_\ell
    \geq \frac{1}{16}  h^{(-\alpha+\varepsilon_0)(\beta-1-\varepsilon_0)} \cdot
    h^{-\alpha/\theta+\varepsilon_0/\theta}
  =\frac{1}{16}
  h^{-\alpha(\beta-(1-1/\theta))} \cdot h^{\varepsilon_0
  (\beta-(1-1/\theta))-\varepsilon_0^2}.
\end{equation*}
\end{fact}

We postpone to prove Fact \ref{t253} and we finish the proof of the lower bound on $f_g(\alpha)$ given Fact \ref{t253}.
\begin{proof}[Proof of \eqref{o126}]
Fix a small $\varepsilon>0$. Choose an $\varepsilon_0$ satisfying \eqref{t215} and
\begin{equation}\label{t266}
  \varepsilon_0<\varepsilon \cdot \frac{\beta-(1-1/\theta)}
  {1+\alpha+\beta-(1-1/\theta)}.
\end{equation}
By Fact \ref{t253} and by \eqref{t262}
\begin{equation*}\label{t263}
  \#\left\{k<2^\ell :
|\Delta_{\ell }^{k }Z|>h^{\alpha}\right\}
 \geq  \frac{1}{16} \cdot
  h^{-\alpha(\beta-(1-1/\theta))} \cdot h^{\varepsilon_0
  (\beta-(1-1/\theta))-\varepsilon_0^2}.
\end{equation*}
Apply Corollary  \ref{t265}  with replacing  $\alpha$ by $\alpha-\varepsilon$. Then using \eqref{t266} we obtain
\[ \begin{aligned}
  \#\left\{
  k<2^\ell :
  |\Delta_{\ell }^{k}Z|>h^{\alpha-\varepsilon}
  \right\}
&<
6c_6(\varepsilon_0)h^{-\alpha(\beta-(1-1/\theta))-\varepsilon_0
+\varepsilon(\beta-(1-1/\theta))}\\
&<\frac{1}{32} \cdot
  h^{-\alpha(\beta-(1-1/\theta))} \cdot h^{\varepsilon_0
  (\beta-(1-1/\theta))-\varepsilon_0^2}
\end{aligned} \]
whenever $\ell $ is large enough.  This yields that
\begin{equation*}%\label{t267}
   \#\left\{k<2^\ell :
|\Delta_{\ell}^{k }Z|\in\left(h^{\alpha},h^{\alpha-\varepsilon}\right)\right\} \\
  >\frac{1}{32} \cdot
  h^{-\alpha(\beta-(1-1/\theta))} \cdot h^{\varepsilon_0
  (\alpha+\beta-(1-1/\theta))-\varepsilon_0^2}.
\end{equation*}

This immediately implies that \eqref{o126} holds.
\end{proof}

\begin{proof}[Proof of Fact \ref{t253}]
Recall that we defined the set $Q_j$  in \eqref{z89} as the $2^\ell$-mesh intervals where $Z_j$ has a loss and that $J_\ell$ defined in \eqref{t255} collects those processes that have the right intensity for our purpose. First we prove that
\begin{equation}\label{z90}
 \# \bigcup\limits_{j\in J_\ell }Q_j > \frac{1}{2}S_\ell ,
\end{equation}
holds almost surely for a sufficiently large $\ell $, where
$$
S_\ell :=\sum\limits_{j\in J_\ell }\#\mathcal{N}_j.
$$
Basically \eqref{z90} means that the intervals where the processes in $J_\ell$ jump do not have too much overlap, that is, at least half of the total number of jumps are kept when taking the union.
%Note that if the sets $Q_j$ were disjoint then we would have $ \# \bigcup\limits_{j\in J_\ell }Q_j=S_\ell $.
%If \eqref{z90} is true then we lose on the possible intersections between $Q_j$'s at most half of $S_\ell $.
On the other hand, by \eqref{t259} (see also \eqref{t260})
\begin{equation}\begin{aligned}\label{z91}
  \#\left(\bigcup\limits_{j\in J_\ell }Q_j\setminus\bigcup\limits_{j\in J_\ell }
  \widetilde{\mathcal{I}}_{j}^{w} \right)
 &\leq   \#\bigcup\limits_{j\in J_\ell }\left(
 Q_j\setminus \widetilde{\mathcal{I}}_{j}^{w}\right)\\
  &\leq
      \sum\limits_{j\in J_\ell }\#\left(
 Q_j\setminus \widetilde{\mathcal{I}}_{j}^{w}
 \right)\\
 &\leq
 (1- \e{-1/w})\sum\limits_{j\in J_\ell }\mathcal{N}_j\\
 &=S_\ell  \cdot (1-\e{-1/w}).
\end{aligned}\end{equation}
Recall $K_\ell$ from \eqref{k-ell}. Using the identity
\begin{equation}\label{z92}
  \bigcup\limits_{j\in J_\ell }\widetilde{\mathcal{I}}_{j}^{w}=
   \bigcup\limits_{j\in J_\ell }Q_j\setminus
   \left[ \bigcup\limits_{j\in J_\ell }Q_j\setminus \bigcup\limits_{j\in J_\ell }\widetilde{\mathcal{I}}_{j}^{w}\right],
\end{equation}
in combination with \eqref{z90} and \eqref{z91} yields that
\begin{equation}\label{z93-immed}
    K_\ell  \geq \left(\frac{1}{2}-(1-\e{-1/w})\right) \cdot S_\ell
     \geq  \frac{ S_\ell}{4}  \geq \frac{1}{4}\#J_\ell\cdot \min\limits_{j\in J_\ell }\mathcal{N}_j
     \end{equation}
  Further, \eqref{t241} and \eqref{e1juli} as well as \eqref{t255} yield that
   \begin{equation}\label{z93}
   K_\ell \ge \frac{1}{4}h^{-(\alpha-\varepsilon_0)(\beta-1-\varepsilon_0)} \cdot
   \frac{\lambda_j}{2} \geq
   \frac{1}{8}h^{-(\alpha-\varepsilon_0)(\beta-1-\varepsilon_0)} \cdot
   h^{-(\alpha-\varepsilon_0)/\theta}.
\end{equation}
So, to complete the proof of Fact \ref{t253} we only need to verify that \eqref{z90} holds.
To do this, imagine that we have $S_\ell $ balls that we throw into $2^\ell $ urns independently. The number $B_\ell $ of non-empty urns stochastically dominates $\#\cup \widetilde{\mathcal{I}}_j^w$ (Since in $\widetilde{\mathcal{I}}_j^w$ we have the extra restriction that the previous jump had to happen relatively long ago). Using that
$S_\ell <10^{-6} \cdot 2^\ell $ (implying that the probability that two balls fall into the same urn is small) it is easy to see that
$$
\sum\limits_{\ell }\mathbb{P}\left\{B_\ell <\frac{S_\ell }{2}\right\}<\infty.
$$
So, by Borel-Cantelli Lemma for $\ell $ large enough, the assertion of Fact \ref{t253} holds.
\end{proof}

%%%%%%%%%%%%%%%%%%%%%%%
%%%%%%%%%itt kezdődik
%%%%%%%%%%%%%%%%%%%%

\section{Upper bound for $f_g$ on region $R_2$}\label{z104}
\subsection{The upper bound}
Let us write $\beta':=\beta-\left(1-\frac{1}{\theta}\right)$. We work now on region $R_2$. That is
\begin{equation}\label{z53}
  \frac{1}{\beta'} \leq \alpha<1+\frac{1}{\beta'}.
\end{equation}
Our aim is to verify that on this region we have
\begin{equation}\label{z66}
f_g(\alpha) \leq 1+\frac{1}{\beta'}-\alpha.
\end{equation}

In this section first we show that on region $R_2$ we have
\begin{equation}\label{z66-osc}
f_g^O(\alpha) \leq 1+\frac{1}{\beta'}-\alpha,
\end{equation}
then we quickly derive the same estimate for $f_g(\alpha)$ on the lower part of this region, on $R_2^\ell$.

We start defining
\begin{equation}\label{jhe}
\mathfrak{J}_{h,\varepsilon}:=\left(-h^{\alpha-\varepsilon},-h^{\alpha+\varepsilon}\right)
\bigcup
\left(h^{\alpha+\varepsilon},h^{\alpha-\varepsilon}\right),
\end{equation}
the sizes of increments/oscillations that have the right absolute value when counting the $N_\ell^\varepsilon(\alpha)$ for $f_g(\alpha)$ or $N_\ell^{\varepsilon, \mathrm{O}}(\alpha)$ for $f_g^{\mathrm O}(\alpha)$, respectively.
First we  prove that for every $\ell $ big enough and $k<2^\ell $ we have
\begin{equation}\label{t273}
  \mathbb{P}\left(\mathrm O^k_\ell Z
  \in \mathfrak{J}_{h,\varepsilon}
  \right)\sim h^{\alpha-1/\beta'}.
\end{equation}
Then by Lemma \ref{o72} we conclude that $N_{\ell }^{\varepsilon, \mathrm{O}}(\alpha)$ (defined similarly as in \eqref{o121}, but with $\mathrm O_{\ell}^k Z$ instead) satisfies
\begin{equation}\label{t274}
  N_{\ell }^{\varepsilon, \mathrm{O}}(\alpha)  \lesssim
  h^{-1+\alpha-1/\beta'},
\end{equation}
where as always $h=2^\ell $. Take  logarithm on both sides and divide by $\log h^{-1}$ to obtain that
\eqref{z66} holds.

In the rest of the section we make these heuristics precise and prove \eqref{t274}. The line of the proof  is the same as the one of the corresponding statement in \cite{Rams2012}.

The rest of the section is organized as follows: First in Section \ref{z68}  we introduce some notation used in this section. In particular we define the notion of a \emph{good loss}. Then in Section \ref{z69} we explain the heuristics of the proof.
As an important technical step of the proof, in  Section \ref{z70} we verify that for every net interval $I_{\ell }^{k}$ there exists a process $Z_j$ with intensity that satisfies $\lambda_{j}^{\theta}\sim h^{-\frac{1}{\beta'}}$ and $Z_j$  has a good loss in  $I_{\ell }^{k}$.
Finally we present the proof of \eqref{z66} in Section \ref{z75}.

For technical reasons, we need to divide $R_2$ into a lower part $R_2^{\ell }$ (defined as $\alpha\le \beta/\beta'$) and an upper part  $R_2^{u}$ (defined as $\alpha > \beta/\beta'$).
We have the upper bound for the \emph{oscillations $O^k_\ell Z$} and thus obtain that $f_g^{\mathrm O}(\alpha)$ satisfies \eqref{z66-osc}.

On the lower part  $R_2^{\ell }$, we derive estimates along the lines so that it will be easy to see that $O_\ell^k Z-|\Delta_\ell^k Z|$ is relatively small compared to $h^\alpha$, hence, we immediately obtain that on this region, $f_g(\alpha)$ satisfies the same upper bound as $f_g^{\mathrm O}$ does, thus we obtain \eqref{z66}.

Unfortunately, on the upper part of the region $R_2^{u}$, the difference between the oscillation and the increment is typically bigger than $h^\alpha$, that is, we have $O^k_\ell Z-|\Delta _\ell^k Z|\gg h^\alpha$. This makes it impossible to transfer our result on $f_g^{\mathrm O}(\alpha)$ to $f_g(\alpha)$ on this region.

We remark however that the `sizes of bursts' are rather determined by the oscillations, not by the increments of $Z$, thus, we believe that this upper bound on $f_g^{\mathrm O}(\alpha)$ is just as relevant for practical purposes as the bound would have been on $f_g(\alpha)$.

\subsubsection{Good losses}\label{z68}

%%%%%%%%%%%%%%%%%%%%%%
\begin{figure}[H]
  \includegraphics[width=9cm]{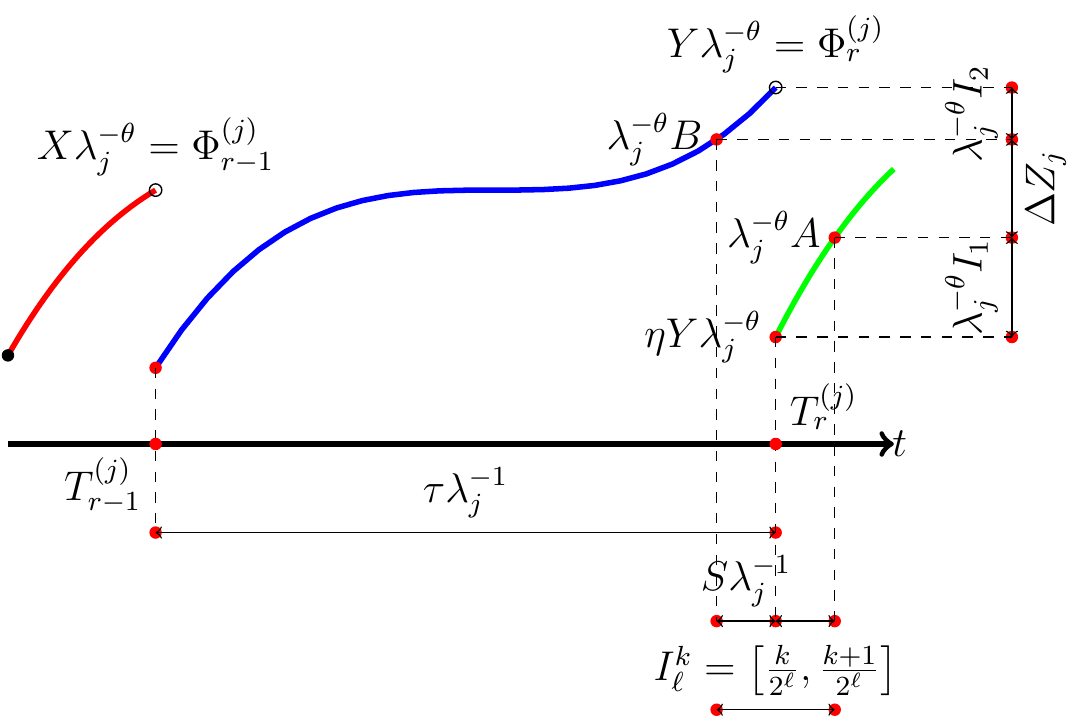}
\caption{Important notations for a process $Z_j$ for the proof of the spectrum on $R_2$.}\label{t289}
\end{figure}Fix  a $j$ and $r$ such that $T_{r}^{(j)}$, the time of r-th loss of $Z_j$ is in $(0,1)$.
The variables defined below are dependent on $j$ and $r$ but in order to simplify the notation we suppress them. Recall that the order of magnitude of a process $Z_j$ with intensity $\lambda_j$ is $\lambda_j^{-\theta}$.
Thus it is reasonable to scale back and define
\begin{equation}\label{t287}
  X:=\lambda_j^{\theta}\Phi_{r-1}^{(j)},\
  \tau:=\lambda_j\tau_{r}^{(j)}
  \mbox{ and }
  Y:=\lambda_j^{\theta}\Phi_{r}^{(j)}=g_X(\tau).
\end{equation}
Moreover,
\begin{equation*}\label{t290}
  A:=\lambda_{j}^{\theta} Z_j\left(\frac{k+1}{2^\ell }\right)
  \mbox{ and }
  B:=\lambda_{j}^{\theta} Z_j\left(\frac{k}{2^\ell }\right)
\end{equation*}
Finally, let
\begin{equation}\label{t291}
  I_1:=Y-B
  \mbox{ and }
  I_2:=A-\eta \cdot Y
\end{equation}
With these notation it is easy to see that the increment of $Z_j$ on the interval $I_\ell^k$ can be written in the form
\begin{equation}\label{t292}
  \Delta Z_j=\lambda_{j}^{-\theta}\left(\left(1-\eta\right)Y
  -
  \left(I_1+I_2\right)\right).
\end{equation}
while the oscillation of $Z_j$ is simply given by
\begin{equation}\label{t292-osc}
  \mathrm O Z_j=\lambda_{j}^{-\theta}\left(1-\eta\right)Y.
\end{equation}
\begin{definition}\label{z71}
Let
\begin{equation}\label{t298}
  G_\delta:=\left\{t:g'_1(t)<\delta\right\}.
\end{equation}
Fix a $j$ and an $r$ such that the $r$-th loss $T_{r}^{(j)}$ of $Z_j$ is in $(0,1)$.
 Let $\mathbf{K}:=(K_1, \dots ,K_5)$, where all components are positive numbers.
We say that $T_{r}^{(j)}$  is $(\mathbf{K},\delta_0)$-regular if
\begin{description}
  \item[(i)] the corresponding three consecutive inter-event times (see \eqref{t197}) all satisfy
\begin{equation}\label{t282}
 \frac{1}{K_1 \cdot \lambda_{j}}
<
\tau_{r-1}^{(j)},\tau_{r}^{(j)},\tau_{r+1}^{(j)}
<
K_2 \cdot \frac{1}{\lambda_{j}},
\end{equation}
\item[(ii)] $\mathcal{L}\mathrm{eb}(G_{\delta_0})<1/K_3$
  \item[(iii)] Further, $X$ as defined in \eqref{t287} satisfies that
\begin{equation}\label{t283}
  X<K_4 ,
\end{equation}
where $K_4$ satisfies
\begin{equation}\label{z58}
  \pi(K_4,\infty )<K_5.
\end{equation}
  \item[(iv)] Moreover we also require that
\begin{equation}\label{t300}
  \frac{\tau}{X^{1/\theta}}\not\in G_{\delta_0}.
\end{equation}
\end{description}
If {\bf(i)-(iv)} hold then we also say that $Z_j$ has a $(\mathbf{K},\delta_0)$ loss in $I_{\ell }^{k}$.
\end{definition}
\begin{fact}\label{z72}We use the notation of Definition \ref{z71}.
  For any $v\in(0,1)$ we can select a vector $\mathbf{K}=\mathbf{K}^{(v)}$ with sufficiently large components and a sufficiently small positive number $\delta_0=\delta_{0}^{(v)}$ such that
 \begin{equation}\label{z73}
\mathbb{P}\left(T_{r}^{(j)}\mbox{ is not $(\mathbf{K}^{(v)},\delta_0^{(v)})$-regular }\right)<v.
\end{equation}
\end{fact}
The proof of Fact \ref{z72} is immediate from the definitions.
\begin{definition}
If $T_{r}^{(j)}$ is called a $v$-good loss if
$(\mathbf{K}^{(v)},\delta_0^{(v)})$-regular.
\end{definition}

The next fact analyses the properties of the increment of $Z_j$ on a $2^\ell$-mesh interval given that $Z_j$ has a $(\mathbf K, \delta_0)$-regular loss there.
 \begin{fact}\label{z61}
   Let $k<2^\ell $ and assume that $Z_j$ has a $(\mathbf{K},\delta_0)$-regular  loss on $I_{\ell }^{k}$. Let $X,Y$ be as in \eqref{t287}. Then
   \begin{equation}\label{z59}
X,Y \geq \frac{c_1}{K_1^\theta}.
\end{equation}
and
\begin{equation}\label{z60}
  Y \leq K_{66}:=K_4 \cdot g_1\left(\frac{K_2 \cdot K_1^\theta}{c_1}\right).
\end{equation}
 \end{fact}
 \begin{proof}
   To verify \eqref{z59} first we introduce
  $x_{-1}:=\lim\limits_{t\downarrow {T_{r-1}^{(j)}}}Z_j(t) $. Then we apply \eqref{z48} and \eqref{t282} in this order to obtain the sequence of inequalities
$$
  X \cdot \lambda_{j}^{-\theta}=g_{x_{-1}}\left(\tau_{r-1}^{(j)}\right)
  \geq   c_1\left(\tau_{r-1}^{(j)}\right)^\theta
   \geq \frac{c_1}{K_1^\theta}\lambda_{j}^{-\theta}.
$$
Applying the same for $\tau_{r}^{(j)}$ and $Y$ instead of $\tau_{r-1}^{(j)}$ and $X$ we complete the proof of \eqref{z59}.

  Now we prove that \eqref{z60} holds.
Using \eqref{z59} and the fact that $g_1$ is monotone increasing we get
as well as the self-similarity property \eqref{t123} we obtain that
\[
  \lambda_{j}^{-\theta}Y=g_{X\lambda_{j}^{-\theta}}\left(\tau\lambda_{j}^{-\theta}\right)  = X\lambda_{j}^{-\theta}
  g_1\left(\frac{\tau\lambda_{j}^{-1}}{X^{1/\theta}\lambda_{j}^{-1}}\right)
  \]
 Using  \eqref{t282}, \eqref{t283} to estimate the right hand side yields the upper bound
\[ \lambda_{j}^{-\theta}Y \le \lambda_{j}^{-\theta}K_4 g_1\left(\frac{K_2 \cdot K_1^\theta}{c_1}\right), \]
which finishes the proof.
 \end{proof}

%%%%%%%%%%%%%

%%%%%%%%%%%%%%%%

\subsubsection{Heuristics of the proof of \eqref{z66} and \eqref{z66-osc}}\label{z69}
In this section we describe the intuitive idea behind the proof of the upper bound of the multifractal spectrum of $Z$ on the region $R_2$.

First, we determine the smallest magnitude of the intensity $\lambda_j$, for which at least one of the $Z_j$s having intensity of this magnitude,  still has a loss in every net interval $I_{\ell }^{k}$. Recall that $\beta'=\beta-(1-1/\theta)$, where $\beta$ is as in \eqref{z4}. We claim that this happens when
we focus on those $j$  for which
\begin{equation}\label{z74}
  \lambda_{j}^{\theta}\sim h^{-1/\beta'}.
\end{equation}
Namely,
if we choose $q$ such that $h^{-1/\beta'}\in(L^{q},L^{q+1})$ but $h^{-1/\beta'}\sim L^q$ then by the definition of $\beta$ we have
$$
\#\left\{j:\lambda_{j}^{\theta}\in (L^{q},L^{q+1})\right\}\sim h^{-((\beta-1)/\beta')}.
$$
This indeed follows from the definition \eqref{z4} or \eqref{o80} of $\beta$, see also \eqref{o81} and \eqref{o83}.
For each of these $j$ the process $Z_j$ has approximately $\lambda_j\sim h^{-1/(\theta \beta')}$ losses on $(0,1)$. So the total number of losses of each of these $j$ is the product
$$
h^{-((\beta-1)/\beta')} \cdot h^{-1/(\theta \beta')}=h^{-1}.
$$
This shows that
the total number of the losses of the union of  $Z_j$ with intensity satisfying \eqref{z74} is approximately $h^{-1}$.
We actually prove that for each net interval $I_{\ell }^{k}$ we can find a $Z_j$ satisfying \eqref{z74} which has  a sufficiently regular loss in $I_{\ell }^{k}$. Then the proof continues as follows: we decompose the oscillation $\mathrm O_{\ell }^{k}Z$ into the oscillation of this $Z_j$ plus the oscillation of all the other $Z_i$, $i\ne j$. Respectively, on the lower part of the region $R_2$ we decompose the increment $\Delta_{\ell }^{k}Z$ into the increment of this $Z_j$ plus the increment of all the other $Z_i$, $i\ne j$. From here, since we have a good control on the oscillation/increment of this particular $Z_j$, we can show that \eqref{t274} holds whatever the oscillations/increments of the remaining $Z_i$s are.
We use the following notion for the rest of this section.

Fix a pair $(\beta,\alpha)\in R_2$. Recall that an $\varepsilon_0$ comes from the regularity of the sequence $(\lambda_j)_{j\ge 1}$, see \eqref{o83}.
  For the rest of the proof we choose an $\varepsilon, \varepsilon_1, \varepsilon_2>0$ in such a way that they satisfy the following inequalities:
   \begin{align}%\label{neplain}
   3\varepsilon&<\alpha-1/\beta',\nonumber \\
 \label{z78} \frac{\varepsilon_0}{\beta'-\varepsilon_0}&<\varepsilon_1<\varepsilon_2<3 \theta^2\varepsilon_0<\varepsilon.
 \end{align}
Note that this choice is possible since we assume \eqref{z53} on $R_2$ and $\beta'\theta^2>1$ holds when $\beta>1$ and when $\beta=1$ then $\beta'=1/\theta$ and $\theta\ge 1$.
\subsubsection{The existence of an appropriate $Z_j$ for an arbitrary net intervals}\label{z70}

In this section we prove that

\begin{lemma}\label{t281}
 Assume the sequence $(\lambda_j)_{j\ge 1}$ is regular. Then, if $\ell $ is large enough then for every $k<2^\ell $ there exists a $j$ such that $T_{r_j}^{(j)}(k,\ell )$ is a $2^{-2/c}$-good loss.
\end{lemma}
Fix $c>0$ as in the statement of the lemma. Let us pick a constant $C>0$ large enough and define the interval with $h=2^{-\ell}$
\begin{equation}\label{z85}
[M_1, M_2]:=\left[h^{-\left(1+\varepsilon_1\right)/\beta'}, C \ell h^{-\left(1+\varepsilon_1\right)/\beta'}\right]
\end{equation}
in such a way that \eqref{o83} is satisfied $c\ell$ many times. That is, $C\ell$ is so large that in the interval $[2^{\ell(1+\varepsilon_1)/\beta'}, C \ell 2^{\ell(1+\varepsilon_1)/\beta'}]$ there is at least $c\ell$ many regular intervals $[L^q, L^{q+1}]$ with at least $L^{q(\beta-1-\varepsilon_0)}$ many $\lambda_j^\theta$-s falling in each of them. Note that when $\ell$ is so large that $C\ell<2^{\ell (\varepsilon_2-\varepsilon_1)/\beta'}$ holds for some $\varepsilon_2\ge \varepsilon_1$, then we have the bound
\begin{equation}\label{newepsilon2}
M_2\le h^{-(1+\varepsilon_2)/\beta'}
\end{equation}
that we shall use later on.
We use now that $L^q\ge 2^{\ell(1+\varepsilon_1)/\beta'}$ to obtain that
\begin{equation}\label{noL}
 \#\left\{j:\lambda_j^\theta \in [M_1, M_2] \right\} \ge c \ell 2^{\ell (1+\varepsilon_1)(\beta-1-\varepsilon_0)/\beta'}
 \end{equation}
Let $\Gamma_\ell$ be the elements of the set in formula \eqref{noL}, and let us partition $\Gamma_\ell$ into disjoint sets immediately as $\Gamma_\ell= \bigcup_{i=1}^{c\ell}\Gamma_\ell^{(i)}$ according to which regular interval $\lambda_j^\theta$ is falling in ($j$s outside the regular intervals can  be assigned arbitrarily to one of these sets).  For an $j\in \Gamma_\ell^{(i)}$, $i\le c\ell$ and  $k\le 2^\ell $ we define the events
$$
E_{k}^{j}:=\left\{Z_j \mbox{ has a loss on }I _{\ell }^{k}\right\},
\qquad
E_\ell^{(i)} :=
\bigcap\limits_{k=1 }^{2^\ell }
\bigcup\limits_{j\in \Gamma_\ell^{(i)} }
E_{k}^{j}.
$$
\begin{fact}\label{z81}
  There exists $\ell _8$ such that for all $\ell >\ell _8$, the event $\bigcap_{i\le c\ell}E_\ell^{(i)} $ holds almost surely. That is, for every $k\le 2^\ell$, there is at least $c\ell$ many processes $Z_j$ with intensity in $[M_1, M_2]$ that have a loss on $I_k^\ell$.
\end{fact}
\begin{proof}
Note that the union of  independent Poisson processes is Poisson process with intensity as the sum of the intensities. Hence
\begin{equation*}\label{z82}
\begin{aligned}
  \mathbb{P}\Big(\Big(\bigcup_{j\in \Gamma_\ell^{(i)}}E_{k}^{j}\Big)^c\Big)
  &=\exp\Big\{-\sum_{j\in \Gamma_\ell^{(i)}}\lambda_j h\Big\}\\
  &\le \exp\left\{ 2^{\ell \frac{(1+\varepsilon_1)(\beta-1-\varepsilon_0)}{\beta'}} 2^{\ell \frac{(1+\varepsilon_1)}{\theta\beta'}} 2^{-\ell}\right\}.
\end{aligned}
\end{equation*}
where we used that the cardinality of $\Gamma_\ell^{(i)}$ is at least as in \eqref{noL} and that the intensity $\lambda_j$ is at least $M_1^{1/\theta}$.
Using now that $\beta-1=\beta'-1/\theta$, the exponent of $2^\ell$ on the rhs of the previous formula simplifies and becomes
$\beta' \theta \varepsilon_1 - \theta \varepsilon_0 (1+ \varepsilon_1):=\delta$, which is positive as long as $\varepsilon_1> \varepsilon_0/(\beta'-\varepsilon_0)$, that we precisely assumed in \eqref{z78}.
To continue the proof, we apply a union bound to obtain that
$$
\mathbb{P}\left((E_\ell^{(i)})^c\right)=\mathbb{P}\Big(\Big(\bigcap\limits_{k=1}^{2^\ell }\bigcup\limits_{j\in \Gamma_\ell^{(i)} }E_{k}^{j}\Big)^c\Big)
   \leq 2 ^\ell \cdot\exp\left\{-2^{\delta \ell }\right\}.
$$
Finally, we apply a union bound again
$$
\mathbb{P}\left(\big(\bigcap_{i=1}^{c\ell}E_\ell^{(i)}\big)^c\right)  \leq  c \ell 2 ^\ell \cdot\exp\left\{-2^{\delta \ell }\right\}.
$$
The right hand side is summable.
Now we apply Borel-Cantelli Lemma to complete the proof of the Fact \ref{z81}.
\end{proof}
\begin{proof}[Proof of Lemma \ref{t281}]
Fix an arbitrary $k<2^\ell$.
By Fact \ref{z81} there is at least $c \ell$ many $j$s with respective losses $T_{r_j}^{(j)}(k,\ell )$ that fall in $I_{k }^{\ell}$.
  By Fact \ref{z72} the probability that $T_{r_j}^{(j)}(k,\ell )$ is not a $2^{-2/c}$-good loss is less than $2^{-2/c}$. Note that these processes are independent. Hence the probability that for all $\ell c$ processes, the loss in $I_k^\ell$ is not a $2^{-2/c}$-good loss is $\left(2^{-2/c}\right)^{\ell c}=2^{-2\ell }$. So, by a union bound, the probability that there exists at least one $k<2^{\ell }$ that does not have a process with a
  $2^{-2/c}$-good loss on it, is less than $2^\ell  \cdot 2^{-2\ell }=2^{-\ell }$. This is summable in $\ell$, so, from Borel-Cantelli Lemma we obtain that for a sufficiently large $\ell _0$, for all $\ell >\ell _0$ we have
\begin{equation}\label{z83}
   \forall k<2^\ell ,\
   \exists  j\in \Gamma_\ell, \
    T_{r_j}^{(j)}(k,\ell ) \mbox{ is a  }2^{-2/c} \mbox{-good loss.}
  \end{equation}
  This finishes the proof.
  \end{proof}

\subsubsection{Deriving the upper bound on $f_g^{\mathrm O}(\alpha)$ and $f_g(\alpha)$}\label{z75}
%%%%%%%%%%%%%%%%%%%%

 To verify \eqref{t273}  first we fix an $\ell $ large enough and an arbitrary $k<\ell $. Using Lemma \ref{t281} we choose an $j\in \Gamma_\ell $ such that for an apropriate $j=j(k,\ell )$ and $r=r(k,\ell ,j)$,
$Z_j$ has a $2^{-2/c}$ good loss $T_{r}^{(j)}\in I_{\ell }^{k}$, where $c$ as in Lemma \ref{t281}. For
$v=2^{-2/c}$  we write
 \begin{equation}\label{z84}
   (\mathbf{K},\delta_{0}):=(\mathbf{K}^{v},\delta_{0}^{v}).
 \end{equation}
 First we decompose $\Delta_{\ell }^{k} Z$ as follows
\begin{equation*}\label{t285}
 \Delta_{\ell }^{k} Z:=\Delta_{\ell }^{k} Z_j+\Delta_{\ell }^{k} Z_{\ne j},
\end{equation*}
where $ \Delta_{\ell }^{k} Z_{\ne j}:=\sum_{q\ne j}\Delta_{\ell }^{k} Z_q$.
On the other hand, for the oscillations we can get upper and lower bounds: a jump of the process $Z_j$ already produces an oscillation of size $\mathrm O Z_j(k,\ell )$, while triangle inequality yields the upper bound:
\begin{equation}\label{t285-2}
O_{\ell }^{k} Z_j\le \mathrm O_{\ell }^{k} Z\le \mathrm O_{\ell }^{k} Z_j+\mathrm O_{\ell }^{k} Z_{\ne j},
\end{equation}
where similarly $\mathrm O_{\ell }^{k} Z_{\ne j}:=\mathrm O_{\ell }^{k} (\sum_{q\ne j} Z_q)$, i.e., this latter is the oscillation of the \emph{sum} of all the other processes. From now on we often suppress $(k,\ell )$ in the rest of the proof. Further, we write $F_{\ne j}$ for the CDF of $\Delta Z_{\ne j}$, and $F_{\ne j}^{\mathrm O}$ for the CDF of $\mathrm O Z_{\ne j}$.

We handle the increments first.
Recall that
 $\varphi_j$ denotes the density function of $\Delta Z_j$. Recall the notation $\mathfrak{J}_{h,\varepsilon}$ from \eqref{jhe}.
Clearly, by a simple convolution,
\begin{multline}\label{t286}
  \mathbb{P}\left(\Delta_{\ell }^{k} Z\in\mathfrak{J}_{h,\varepsilon}\right)
  = \mathbb{P}\left(\Delta_{\ell }^{k}Z_j
  +
  \Delta_{\ell }^{k}Z_{\ne j}\in\mathfrak{J}_{h,\varepsilon}\right)
  =\int\limits_{a\in\mathbb{R}}
  \int\limits_{y\in a+\mathfrak{J}_{h,\varepsilon}}
  z_j(y)dy dF_{\ne j}(a)
   \\
   \leq
   2\max \limits_{a}\int\limits_{a}^{a+h^{\alpha-\varepsilon}}z_j(y)dy
   =
   2\max \limits_{a}
   \mathbb{P}\left(\Delta Z_j\in a+\left(0,h^{\alpha-\varepsilon}\right)\right).
\end{multline}
%%%%%%%%%%%%%%%%%%%%
Now we assume that
\begin{equation}\label{t288}
 1 \leq  \alpha \cdot \beta' \leq\beta.
\end{equation}
This defines the sub-region of $R_2$ that we denote by $R_2^{\ell }$, see Figure \ref{o60}.
Clearly, the right hand side of \eqref{t288} always holds when $\alpha<1$. Recall the notations $I_1, I_2$ from \eqref{t291}. Using \eqref{t238}, Lagrange Mean Value Theorem, \eqref{t288} and \eqref{z59}, after a somewhat longish but elementary calculation we obtain that
\begin{equation}\label{t293}
  \lambda_{j}^{-\theta}\left(I_1+I_2\right) \leq
  h^{\beta/\beta'}
  \ll
  h^{\alpha}\ll \lambda_j^{-\theta}Y.
\end{equation}
This yields an estimate on the second term in \eqref{t292}.
Hence, using \eqref{t292} combined with \eqref{t293}, we obtain that
\begin{equation}\label{inc-to-osc} \mathbb{P}\left(\Delta Z_j\in a+\left(0,h^{\alpha-\varepsilon}\right)\right)\approx
 \mathbb{P}\left((1-\eta)\lambda_j^{-\theta}Y\in a+\left(0,h^{\alpha-\varepsilon}\right)\right).
 \end{equation}
By the fact that $\lambda_j^\theta\in (M_1, M_2)$ from \eqref{z85} with  the bound on $M_2$ as in \eqref{newepsilon2} we obtain
\begin{equation}\begin{aligned}\label{t294}
  \mathbb{P}\left(\Delta Z\in\mathfrak{J}_{h,\varepsilon}\right)& \leq
  2\max\limits_a
   \mathbb{P}\left(Y\in
  a
  +
   \Big(
0,\frac{2}{1-\eta}\lambda_{j}^{\theta}h^{\alpha-\varepsilon}
   \Big) \right) \\
    &\leq
    2\max\limits_a
   \mathbb{P}\left(Y\in
  a
  +
   \left(
0, h^{\alpha-\frac{1}{\beta'}(1+3\theta^2\varepsilon_0)-\varepsilon}
   \right) \right)\\
    &=
    2\max\limits_a
   \mathbb{P}\left(g_X(\tau)\in
  a
  +
   \left(
0,h^{\alpha-\frac{1}{\beta'}(1+3\theta^2\varepsilon_0)-\varepsilon}
   \right)\right),
\end{aligned}\end{equation}
where $\tau$ is a  truncated $\mathrm{Exp(1)}$ random variable taking values from the interval $\left(\frac{1}{K_1},K_2\right)$
and the distribution of $X$ is $\pi(\cdot|\ (K_1^{-\theta},K_4))$. This is so, because we assumed that $Z_j$ has a good loss in $I_{\ell }^{k}$.
Let $\widetilde{\varphi}(x)$ and $\widetilde{f}(t)$ be the density functions of $X$ and $\tau$ respectively. Moreover, for a fixed $x$ let $y\mapsto \psi_x(y)$ be the density function of $Y=g_x(\tau)$ conditioned on $X=x$. That is by \eqref{z56}
$$
\psi_x(y):=\frac{\widetilde{f}\left(g_{x}^{-1}(y)\right)}
{g'_x\left(g_{x}^{-1}(y)\right)}=
\frac{\widetilde{f}\left(g_{x}^{-1}(y)\right)}
{x^{1-1/\theta} \cdot g'_1\left(\frac{\tau}{x^{1/\theta}}\right)}
$$
By the choice of $j$ we know that on the one hand, $g'_1\left(\frac{\tau}{x^{1/\theta}}\right)>\delta_0$ on the other hand, $x>K_1^{-\theta}$. From these we obtain that
\begin{equation}\label{z57}
  \psi_x(y)  \leq \frac{2K_1^{\theta-1}}{\delta_0}.
\end{equation}
Then the right hand side of \eqref{t294} without the maximum can be estimated as follows
\begin{equation}\begin{aligned}\label{t295}
   &\int\limits_{x=K_1^{-\theta}}^{K_4}
   \mathbb{P}\left(g_x(\tau)\in
  a
  +
   \left(
0,h^{\alpha-\frac{1}{\beta'}(1+3\theta^2\varepsilon_0)-\varepsilon}
   \right)\right)
 \widetilde{\varphi}(x)  dx\\
& \quad=
 \int\limits_{x=100^{-\theta}}^{K_4}\Bigg(
 \underbrace{\int\limits_{y=a}^{a+h^{\alpha-\frac{1}{\beta'}
 (1+3\theta^2\varepsilon_0)-\varepsilon}}
 \psi_x(y)dy}_{:=\mathcal{I}(a,x)}
\Bigg) \widetilde{\varphi}(x)  dx.
\end{aligned}\end{equation}
Our goal is to verify that there is a constant $c_{55}$ such that
\begin{equation}\label{t296}\forall a,\ \forall x:\quad
  \mathcal{I}(a,x) <c_{55} \cdot h^{\alpha-\frac{1}{\beta'}
 (1+3\theta^2\varepsilon_0)-\varepsilon}
 <c_{55} \cdot h^{\alpha-\frac{1}{\beta'}(1+\varepsilon)
 -\varepsilon}
\end{equation}
This immediately follows from \eqref{z57} and the lower bound on $\varepsilon$ in \eqref{z78}.
Note that this estimate is independent of $x,a$ and that $\widetilde \varphi_x$ in \eqref{t295} is a density function thus it integrates to $1$. Thus we obtain the upper bound
\begin{equation}\label{c55}
 \mathbb{P}\left(\Delta Z\in\mathfrak{J}_{h,\varepsilon}\right) \le  c_{55} \cdot h^{\alpha-\frac{1}{\beta'}(1+\varepsilon)
 -\varepsilon}
\end{equation}
This establishes the upper bound on $f_g(\alpha)$ in the region $R_2^{\ell }$.

We remark that we used the crucial estimate \eqref{t293} that is only valid on $R_2^{\ell }$ but not on $R_2^{u}$. This made it possible to move from the estimate on $Y$ to the estimate on $\Delta Z$ in \eqref{inc-to-osc}.

Now our goal is to modify the calculations above to hold for the oscillations, for the whole region
$R_2^{\ell }$. Again, we emphasize that on $R_2^{\ell }$ we have $\alpha > \beta/\beta'$, thus the estimate \eqref{t293} is not valid, and as a result \eqref{inc-to-osc} fails to hold. Thus we cannot control the increments, resulting in the lack of a bound on $f_g(\alpha)$ on this region.

On the other hand, for the oscillations we have the inequalities in \eqref{t285-2}.
Even though the convolution argument is no longer valid, by a similar argument as in \eqref{t286}, one can  still show using  \eqref{t285-2} that
\begin{equation*}%\label{t286-2}
  \mathbb{P}\left(\mathrm O Z\in\mathfrak{J}_{h,\varepsilon}\right)
   \leq
   2\max \limits_{a}
   \mathbb{P}\left(\mathrm O Z_j\in a+\left(0,h^{\alpha-\varepsilon}\right)\right).
\end{equation*}
From here on, the calculation is similar as for the increments: one notes that $\mathrm O Z_j=(1-\eta)Y \lambda_j^{-\theta}$ (see \eqref{t292-osc}), and the calculation in \eqref{t294}, \eqref{t295} and \eqref{t296} also remain valid for $\mathbb P(\mathrm O Z \in \mathfrak{J}_{h,\varepsilon})$.
Thus we obtain as in \eqref{c55} that
\[  \mathbb{P}\left(\mathrm O Z\in\mathfrak{J}_{h,\varepsilon}\right) \le  c_{55} \cdot h^{\alpha-\frac{1}{\beta'}(1+\varepsilon)
 -\varepsilon}
 \]
This establishes the upper bound on $f_g^{\mathrm{O}}(\alpha)$ in the region $R_2$ (both on $R_2^{u}$ as well as on $R_2^{\ell }$.

%%%%%%%%%%%%%%%%%%%%%%%
%%%%%%%itt a vége
%%%%%%%%%%%%%%%%%%%%

%%%%%%%%%%%%%%%%%%%%%%%%%%
%%%%%%%%%%%%%
%%%%%%%%%%%%%%%%%

\section{Implications of the Results}
\label{sec:implications}

In this paper we provided the multifractal spectra for one of the most widespread TCP versions of the Internet, i.e., for the TCP CUBIC, which is the default TCP version in the Linux world. We have also compared our results with the results obtained for TCP Reno in~\cite{Rams2012}. Based on our results the following conclusions can be made from the point of view of Internet traffic theory.

%in Theorem \ref{o79}.

%In this Section we interpret this result from the internet traffic theory point of view and also compare this result with the result obtained for TCP Reno in~\cite{Rams2012}.

The multifractal spectrum $f(\alpha)$ can provide a rich characterization of traffic burstiness. Intuitively, $f(\alpha)$ captures how frequently a value $\alpha$ is found. Heuristically speaking, $\alpha$ describes the magnitude of the burst as a power of the time it lasts, on a small time scale. Hence, values for $\alpha<1$ indicate bursty behavior. As a consequence, the values and the shape of $f(\alpha)$ in the range of $\alpha<1$ have the primary importance for the evaluation of traffic burstiness.

The first conclusion follows from Theorem \ref{z30} is that \textit{TCP CUBIC traffic is a bursty traffic} since $f(\alpha)>0$ for all values where $\alpha<1$. This finding is in line with the analysis of TCP CUBIC traces measured in Internet. The importance of our result is that \textit{we have provided the theoretical proof why TCP CUBIC traffic is bursty}.

Our second conclusion can be made if we compare the multifractal spectrum of TCP CUBIC with the result obtained for TCP Reno in~\cite{Rams2012}, see Theorem III.4 in~\cite{Rams2012} as we have carried out in \secref{sec:comparison}. From Corollary~\ref{z26} we see that $f_{Reno}(\alpha)> f_{CUBIC}(\alpha)$ for $\alpha<\frac{9}{10}$ and $f_{Reno}(\alpha)=3f_{CUBIC}(\alpha)$ for $\alpha<\frac{1}{2}$. Please note that this effect is dramatic from the point of view of burstiness since this difference means that the number of dyadic intervals of size $\Delta X$ behaves as $(\Delta X)^{-f(\alpha)}$. As a practical conclusion we can say that the importance of this observation is that \textit{we have theoretically proved that TCP CUBIC traffic is less bursty than TCP Reno}. It is also a good indication why besides many other reasons (faireness, scalability, etc.) TCP CUBIC has been a good choice for being the default version in the Linux world.

As a performance implication of our results regarding queueing performance of TCP CUBIC traffic we refer to our earlier results where we gave the queue tail asymptotic of a single queueing model with general multifractal input~\cite{Dand2003}.

\bibliographystyle{plain}
\bibliography{biblo_q6}

%\end{IEEEbiography}


%\newpage


%%%%%%%%%%%%%%%%%%%
%%%%%%%%%%%%%%%%%%%%%%

\end{document}